\documentclass{amsart}
\usepackage{amssymb}
\usepackage{booktabs}
\usepackage{mathrsfs}
\usepackage{mathtools}
\usepackage{psfrag}
\usepackage{xr-hyper}
\usepackage[colorlinks]{hyperref}
\usepackage{fourier} 
\makeatletter
\@namedef{subjclassname@2020}{%
  \textup{2020} Mathematics Subject Classification}
\makeatother

\subjclass[2020]{ 37A05, 37A25, 37C30, 37D20, 37E05, 46F05, 26A16} 

\keywords{ distribution, Birkhoff sums, expanding maps, Anosov diffeomorphism, Zygmund, Log-Lipschitz, invariant distribution}

\title[Birkhoff sums as distributions I: Regularity ]{Birkhoff sums as distributions I: \\ {\small Regularity}}

\author[C.  Grotta-Ragazzo]{Clodoaldo Grotta-Ragazzo}
\address{ Instituto de Matem\'atica e Estat\'istica,  Universidade de S\~ao Paulo, Rua do Mat\~ao, 1010, Cidade Universit\'aria, S\~ao Paulo-SP, CEP 05508-090, Brazil }
\email{ragazzo@usp.br} 
\author[D.  Smania]{Daniel Smania}
\address{Departamento de Matem\'atica, Instituto de Ci\^encias Matem\'aticas e de Computa\c{c}\~ao (ICMC), Universidade de S\~ao Paulo (USP), Avenida Trabalhador S\~ao-carlense, 400, S\~ao Carlos-SP, CEP 13566-590,  Brazil.}
\email{smania@icmc.usp.br} 
\urladdr{\url{https://sites.icmc.usp.br/smania/}}
\thanks{C.G.R is partially supported by FAPESP, Brazil grant 2016/25053-8. D.S. was supported by FAPESP Projeto Tem\'atico 2017/06463-3 and Bolsas de Produtividade em Pesquisa CNPq 306622/2019-0 and  311916/2023-6, and CNPq Universal 430351/2018-6. 
}

  \newtheorem{theorem}[equation]{Theorem}

\newtheorem{lemma}[equation]{Lemma}

\newtheorem{proposition}[equation]{Proposition}

\newtheorem{mainthm}{Theorem}

\renewcommand{\theequation}{\thesection.\ifnum\value{subsection}=0 1\else \arabic{subsection}\fi.\arabic{equation}}
\newcommand\numberthis{\addtocounter{equation}{1}\tag{\theequation}} 

\theoremstyle{remark}

\newtheorem{rmk}[equation]{\bf Remark}
\newcommand{\erm}{{\rm e}}

\newcommand{\cinf} {\hbox{\rm C}^\infty}

\newcommand{\calF}{{\mathcal{F}}}

\newcommand{\calS}{{\mathcal{S}}}
\newcommand{\ap} {\alpha}

\newcommand{\dt} {\delta}

\newcommand{\Gm} {\Gamma}
\newcommand{\gm} {\gamma}
\newcommand{\Lb} {\Lambda}

\newcommand{\om} {\omega}

\def\Z{{\mathbb Z}}

\def\R{{\mathbb R}}
\def\Tm{{\mathbb T}}

\def\D{{\mathbb D}}

\DeclarePairedDelimiter{\floor}{\lfloor}{\rfloor}

\usepackage{constants} 
\newcommand{\secdot}[1]{\arabic{#1}}
\newconstantfamily{c}{
symbol=\lambda,
format=\secdot,
}
\renewconstantfamily{normal}{
symbol=C,
format=\secdot,
}

\newconstantfamily{e}{
symbol=\theta,
format=\arabic,
}

\newconstantfamily{A}{
symbol=A,
format=\arabic,
}

\newconstantfamily{G}{
symbol=G,
format=\arabic,
}

\newcommand{\Cll}[2][normal]{\Cl[#1]{#2}}
\newcommand{\Crr}[1]{\Cr{#1}}
\newcounter{change}

\usepackage[colorinlistoftodos, textwidth=4cm]{todonotes}

\makeatletter
\providecommand\@dotsep{5}
\renewcommand{\listoftodos}[1][\@todonotes@todolistname]{%
  \@starttoc{tdo}{#1}}
\makeatother

\hypersetup{ colorlinks=true, pdftitle={Birkhoff sums as distributions:  regularity}, pdfsubject={ Birkhoff sums}, pdfauthor={Clodoaldo Ragazzo, Daniel Smania},pdfkeywords={  distribution, Birkhoff sums, expanding maps, Anosov diffeomorphism, Zygmund, Log-Lipschitz, invariant distribution}}

\begin{document}

\begin{abstract}  We study Birkhoff sums as distributions. We obtain regularity results on  such distributions for various dynamical systems with hyperbolicity, as  hyperbolic linear maps on the torus and  piecewise expanding maps on the interval.  We also give some applications, as the study of advection in discrete dynamical systems. \end{abstract}

\maketitle

\setcounter{tocdepth}{2}
\tableofcontents




\section{Introduction}

Consider a measurable dynamical system $f\colon X\rightarrow X$,  where $X$ is a measure space endowed with a reference measure $m$ and such that $f$ has an invariant probability  $\mu$ that is absolutely continuous with respect to $m$.  Let $\phi\colon X\rightarrow \mathbb{C}$  be a measurable observable. One of the main goals in the study of the ergodic theory is to study the statistical properties of the sequence of variables
$$\phi, \phi\circ f , \phi\circ f^2, \dots$$
The Birkhoff ergodic theorem, for instance, says that 
$$\lim_N \frac{1}{N} \sum_{k=0}^N\phi\circ  f^k(x)$$ 
converges for $\mu$-almost every point $x$. 
On the other hand if we  consider the {\it Birkhoff sum }
$$\sum_{k=0}^\infty \phi\circ  f^k.$$ 
then in very regular situations (piecewise expanding maps and Anosov diffeomorphisms) this sum  may  {\it not } converge almost everywhere. Suppose now that $X$ is a manifold with a  measure $m$. Then in many well-known cases  we have that 
$$\sum_{k=0}^\infty \int \psi \phi\circ  f^k \ dm.$$ 
{\it does}  converge when $\psi$ and $\phi$ are regular  enough. This allows us to define
$$\int  \psi \sum_{k=0}^\infty \phi\circ f^k \ dm= \sum_{k=0}^\infty \int \psi \phi\circ  f^k \ dm$$
for $\psi\in C^\infty$, in such way that the Birkhoff sum induces a {\it distribution}. 
We should see the right hand-side of this equation as the {\it definition } of the left-hand side, which  is not a proper Lebesgue integral.

Our goal is to study this distribution's regularity  for several dynamical systems with strong hyperbolicity: maps with exponential decay of correlations, linear Anosov maps, and piecewise expanding maps on the interval. We  give  three main motivations to this study. 
\label{int}

\subsection{Deformations of Dynamical Systems}
A {\it deformation} of a dynamical system $f_0\colon M \to M$ is a smooth family 
$f_t\colon M \to M$, with $t \in (-\epsilon, \epsilon)$, such that $f_t$ is conjugate to $f_0$ for every $t$. That is, there exist homeomorphisms $h_t\colon M \to M$ such that
$$
    h_t \circ f_0 = f_t \circ h_t.
$$

For certain types of maps $f_0$ exhibiting hyperbolicity (e.g., expanding maps, piecewise expanding maps, and Anosov diffeomorphisms), the conjugacies $h_t$ are rarely smooth. In fact, they are often singular with respect to the Lebesgue measure, meaning they map a set of full Lebesgue measure to a set of zero Lebesgue measure. This poses a significant challenge in studying such conjugacies. However, for one-dimensional dynamics (maps acting on an interval or the circle), and even for Anosov diffeomorphisms in arbitrary dimensions, the maps
$$
    t \mapsto h_t(x)
$$
are smooth. This enables the study of {\it infinitesimal deformations}, defined by
$$
    \alpha(x) = \partial_t h_t(x)|_{t=0}.
$$

Investigating the regularity of the function $\alpha$ provides precise insights into the regularity of the conjugacies $h_t$ and their dependence on the parameter $t$. For piecewise expanding maps on an interval, this approach has been developed in C-R and S. \cite{dois}. We believe, however, that it can be extended to much broader contexts. The connection to this paper lies in the fact that $\alpha$ is generally not differentiable. Nevertheless, it has derivatives in the {\it sense of distributions}, and {\it its derivative is a Birkhoff sum}. Consequently, $\alpha$ can be seen as a {\it primitive} of a Birkhoff sum. In this work, we conduct a detailed study of primitives of Birkhoff sums for piecewise expanding maps in one dimension. See Section \ref{main1d} for the main results. In particular, we analyze the statistical properties of these primitives.

\subsection{Invariant Distributions}
An {\it invariant distribution} of a dynamical system $T$ is a distribution $\psi$ that satisfies
$$
    \langle \psi, \phi \circ T \rangle = \langle \psi, \phi \rangle
$$
for all test functions $\phi$. A familiar example of an invariant distribution is a $T$-invariant probability measure. One might ask whether these are the only possible examples. Avila and Kocsard \cite{ak} (see also Navas and Triestino \cite{navas}) proved that for a $C^\infty$ diffeomorphism of the circle with irrational rotation, the unique invariant probability measure is also its {\it only} invariant distribution (up to multiplication by a constant). 

For piecewise expanding maps on an interval, the situation is markedly different. By utilizing Birkhoff sums as distributions, we construct invariant distributions that are {\it not} (complex-valued) measures. See Section \ref{maininv} for the main results.

 \subsection{Advection and cohomological equations  in discrete dynamics} The mathematical problems to be considered in this
 paper are also  motivated by the following physical question.

Let $M$ be a $C^\infty$ manifold that represents 
the state space of a discrete-time 
 dynamical system $f:M\to M$, where $f$ is  a
$C^\infty$ diffeomorphism. 
The manifold is endowed with 
 a $C^\infty$ volume form $m$, which may be that 
associated to a Riemannian metric on $M$. Let  
$J:M\to \R$ be  the Jacobian determinant of $f$ with 
respect to $m$, $J$ is 
supposed positive ($f$ is orientation preserving).
Notice that  $m$
  can be invariant under $f$, $J=1$, or not, $J\ne 1$.  Assume that a  fluid lies on $M$
and that  the dynamic changes 
the position of the fluid particles, the particle  initially at 
$x_0\in M$ moves to $x_1=f(x_0)$. Several extensive 
 physical quantities may be  associated with a fluid:
mass, internal energy, the mass of a diluted chemical substance,
electric charge, etc.
Each of these properties is characterized by a density function 
$\rho$ with respect to the volume form $m$ such that 
$\int_\Omega\rho \ dm$ gives the amount of the property 
 inside the region $\Omega$. For simplicity 
we assume that $\rho$ is the electric charge  density
or just the charge density. 
Suppose that charge is
 advected by the dynamics (there is no charge diffusion 
between neighboring fluid particles). In this case if there 
is no  charge input-output  to the system
then  advection yields: if   $\rho_j:M\to\R$
is the density  at time $j$ then 
$f_{*}(\rho_jm)=\rho_{j+1}m$, for every $i\in \mathbb{Z}$,   implies that the density 
 at time $j+1$ is 
\begin{equation} \label{rhoh}
\rho_{j+1}=L \rho_j\end{equation} 
where $L\colon L^1(m)\rightarrow L^1(m)$ is the Ruelle-Perron-Frobenious operator
\begin{equation} (L\rho)(x)= \frac{\rho\circ f^{-1}(x)}{J\circ f^{-1}(x)}.
\end{equation}
Note that 
\begin{equation} (L^{-1}\rho)(x)= J \cdot ( \rho\circ f).
\end{equation}
More generally, at each time,  charge can be added to or
 subtracted from the fluid  by means of a  $C^\infty$
distribution of sources and sinks
 that is supposed to be  time-independent.
In this case  the form   $\rho_jm$ at time $j$
is mapped to 
$\rho_{j+1}m=f_{*}(\rho_jm+Rm)$ at time $j+1$,
where
$R:M\to \mathbb{R}$ is the density function of sources
 (wherever $R>0$) and of sinks (wherever $R<0$). 
These relations imply  
\begin{equation}
 \rho_{j+1}=L(\rho_j+R)
\label{coceq}
\end{equation}
Given a  $C^\infty$ function $\phi:M\to\R$ with compact support, 
 $\phi\in C^\infty_c(M)$, a measurent of  charge
 at time $j$ is defined as 
\begin{equation}
Q_j(\phi)=\int_M\phi \rho_j \ dm
\label{Qk}
\end{equation}
If $\phi$ is positive and $\int_M\phi \ dm=1$, then 
  $Q_j(\phi)$ is the average charge density 
with respect to the measure $\phi m$. In this way, a measurement
is a linear functional in $C^\infty_c(M)$
and define a distribution in the sense of Schwartz,
 $Q_j\in\mathcal{S}^\prime(M)$ (see  H\"ormander \cite{horm1}).
 The questions we are interested
in are:
\begin{itemize}
\item[({\it i})] Do the limits 
\[
\lim_{j\to-\infty}Q_j(\phi)=u_\alpha(\phi)\qquad\text{and}\qquad 
\lim_{j\to +\infty}Q_j(\phi)=u_\om(\phi)
\]
exist for all functions $\phi\in C^\infty_c(M)$?
\item[({\it ii})] If the limits exist, 
then $u_\alpha$ and $u_\om$ define
 distributions in $\mathcal{S}(M)$. What is the regularity
of   $u_\alpha$ and $u_\om$? 
\end{itemize} 

It is easy to see that 
\begin{align*}
\rho_j&= L^j\rho_0 + \sum_{i=1}^j L^i(R). \numberthis  \label{coc}\\
\rho_{-j}&=L^{-j}\rho_0   - \sum_{i=0}^{j-1} L^{-i}(R) 
\end{align*}
Since we are mostly interested in the component associated to $R$, from now on we assume $\rho_0=0$. 
The fixed point equation associated to equation 
 (\ref{coceq}) is
\begin{equation}
L^{-1} \rho-\rho=R,  \label{cohom}
\end{equation}
 The solutions to that are  invariant charge densities
(if $R=0$ then it defines an invariant measure for $f$). 
An integrable function $\rho$ 
is a weak solution to equation (\ref{cohom}) 
if for every test function $\phi\in C^\infty(M)$ the following
identity holds
$$
\int_M\rho \cdot \phi\circ f^{-1} \ dm-\int_M\rho\phi dm=\int_MR\phi dm
$$
This definition extends naturally to distributions.
A distribution $u\in \mathcal{S}^\prime(M)$ is a weak solution to
equation (\ref{cohom}) if 
for every test function $\phi\in C^\infty(M)$ the following
identity holds
\begin{equation}
u(\phi\circ f^{-1})-u(\phi)=\int_MR\phi dm  \label{wcohom}
\end{equation}
Suppose that the limit $u_\omega(\phi)$ in question ({\it i}) 
exists for any $\phi$. Then, from  (\ref{coc})
$$\lim_{j\to+\infty}\int_M L^j(R)\phi \ dm=0.$$
Using the definition of $u_\omega$ a computation shows that 
\[
Q_j(\phi\circ f^{-1})-Q_j(\phi)-\int_MR\phi \  dm=
-\int_M L^j(R) \phi \ dm,
\]
and taking the limit as $j\to\infty$ we conclude  that $u_\om$
is a weak solution of the cohomological equation. The same result
holds for $u_\ap$. Therefore, if  the $\omega$-limit and the 
$\alpha$-limit (in a weak sense) exist, then these limits
are weak solutions of the corresponding cohomological equation.

From now on, we suppose that  $M$ is compact. 
Then  integration of  both sides of equation 
(\ref{coceq}) over $M$  with respect to the volume form 
$\mu$  gives $$Q_{j+1}(1)=Q_j(1)+\int_MR \ dm.$$ Therefore,
 question ({\it i}) may have a positive answer only if 
\begin{equation}
\int_MR \ dm=0,\label{hR}
\end{equation}
which means that the total amount of charge added to 
the system at each time is null. 

The limits in question ({\it i}) do not exist unless the 
dynamics of $f$ is complex enough. For instance, if $f$ is 
the identity map then $u_\alpha$ and $u_\omega$ do not 
exist for any function $R\ne 0$. Suppose that $m$ is invariant
under $f$. Then $J=1$ and (\ref{cohom}) is the classical {\it Livsic cohomological equation}
$$\rho\circ f -\rho = R$$
and  for $\rho_0=0$
\begin{align*}
\rho_j&= \sum_{i=1}^j  R\circ f^{-j}. \\
\rho_{-j}&=  - \sum_{i=0}^{j-1} R\circ f^j 
\end{align*}

 If $u_\omega$ exists for every  function $R \in \cinf(M)$ satisfying $\int_M R \ dm=0$ then $$\lim_{j\to\infty}\int_M(R\circ f^{-j})\phi \ dm=0$$ for
all functions $R$ and $\phi$ in $\cinf(M)$. This implies  the decay of
 correlations for any pair of functions in $\cinf(M)$, 
which is equivalent to $(f,m)$ to be  mixing. 
Therefore, if $m$ is invariant
under $f$, the limits $u_\ap$ and $u_\om$ exists for any given 
function $R$ only if $(f,m)$ is mixing (this seems to be
a natural physical condition for the existence of an equilibrium
once we have neglected molecular diffusion). On the other hand,
if $(f,\mu)$ is mixing and the decay of correlations is fast
enough so that
\[
\sum_{j\in \mathbb{Z}} \left|\int_M(R\circ f^j)\phi \ dm\right|<\infty
\] 
for any functions $R$ and $\phi$ in $\cinf(M)$ with
$\int_M R\ dm=0$, then the limits $u_\omega$ and $u_\alpha$ 
 exist. Indeed we have that 
 \begin{align*}
u_\alpha &= \sum_{i=1}^{+\infty}   R\circ f^{-j}. \\
u_\omega&=  - \sum_{i=0}^{+\infty} R\circ f^j.
\end{align*}
{\it as distributions}.  

\section{Main results}

\subsection{Dynamics with exponential decay of correlations} Our first  result gives weaker regularity results than the ones we obtain later for hyperbolic linear maps  on the torus and piecewise expanding maps on the interval. However, it is remarkable  that its  only assumption is exponential decay of correlations for H\"older observables, which  has been proved for a wide variety of dynamical systems.

\begin{mainthm}
\label{main} Let $M$ be a $C^r$ compact manifold, with $r\geq 1$, $f\colon M \rightarrow M$ be a measurable function with measurable inverse, and $\mu$ be a smooth volume form invariant under
$f$. Suppose that for a $\mu$-integrable  function  $R\in L^\infty(M)$ 
satisfying  $\int_MR \ d\mu=0$ and for  
$\phi\in C^r(M)$ the exponential decay of correlations 
\begin{equation}
\left|\int_MR(f^j(x))\phi(x)\ d\mu(x)\right|\le \Cll{c1} \erm^{-\Cll{c2}|j|}
|\phi|_{{\rm C}^\gm}, \quad j\in\Z \label{dec1}
\end{equation}
holds, where $\Crr{c1}>0$ and $\Crr{c2}>0$ are constants that  depend neither 
on $j$ nor on  $\phi$ and $|\phi|_{{\rm C}^\gm}$, with $\gm>0$, is the 
usual H\"older norm of $\phi$. Then the Birkhoff sums $u_\alpha$ and 
$u_\omega$ given by 
\[
u_\om=\sum_{-\infty}^{j=-1}
R(f^{j}x)\quad\text{and}\quad
u_\ap=-\sum_{j=0}^\infty R(f^{j}x).
\]
belong to the logarithm Besov space $B^{0,-1}_{\infty,\infty}$.
Moreover they are weak solutions  of  the cohomological 
equation $$R=u\circ f-u.$$
\end{mainthm}

Logarithmic Besov spaces  $B^{s,b}_{p,q}$ are a generalization classical Besov spaces (see Section \ref{defsec}).

\begin{rmk} If $f$ is not invertible, we can obtain a similar result for $u_\ap$ assuming (\ref{dec1}) for $j\geq 0$.
\end{rmk}

A volume-preserving linear Anosov  on the  $n$-dimensional torus  $ \Tm^n=\R^n/\Z^n$ is a  $C^\infty$-diffeomorphism   $f\colon \Tm^n \rightarrow \Tm^n$ defined by  $fx=Mx$, 
where $M$ is a hyperbolic $n\times n$-unimodular matrix. The most famous example is the Arnold's Cat map, an Anosov map  obtained taking 
$$M=\left(\begin{array}{cc}
2&1\\1&1
\end{array}\right).
$$

\subsection{Hyperbolic Linear maps on the Torus} 

For hyperbolic linear maps  we  can use Fourier analysis methods to obtain 

\begin{mainthm} \label{sigma} For every $R\in C^\beta(\mathbb{T}^n)$, with $\beta > n/2$, such that 
$$\int R \ dm=0,$$
where $m$ is the Haar measure of  $\mathbb{T}^n$. Consider the Birkhoff sums
\begin{align*}
u_\ap &=-\sum_{j=0}^\infty R\circ f^j
\end{align*} 
and
\begin{align*}
u_\om &=\sum^{\infty}_{j=1}
R\circ  f^{-j}.
\end{align*}
Then  $u_\ap$ and $u_\om$ are well-defined as  distributions and $u_\om, u_\ap\in \Lambda^0$, where $\Lambda^0$ is a Zygmund space.  Moreover they are both weak solutions of the cohomological equation
\begin{equation}\label{ce56} R=u\circ f-u.\end{equation} 
\end{mainthm} 

Zygmund spaces $\Lambda^s$ are introduced  in Section \ref{defsec}.

\subsection{Piecewise expanding maps: Primitives of Birkhoff  sums}\label{main1d}

Here we give a fairly complete  picture of the regularity of Birkhoff sums for piecewise expanding one-dimensional maps. One of the advantages of the one-dimensional setting is that one can easily define the {\it primitive} of a Birkhoff sum. Let $I=[a,b]$ and  $f$ be a $C^{1+BV}$ piecewise expanding map on $I$. We are going to see that if $\phi\in L^\infty(m)$ is orthogonal to the densities of all absolutely continuous invariant  probability measures of  $f$ then 

$$\psi(x)=   \int  1_{[a,x]}\cdot \Big( \sum_{k=0}^\infty  \sum_{j=0}^{p-1} \phi\circ f^{kp+j} \Big)\ dm$$
is a well-defined function (for an appropriated $p$ that depends only on $f$) and 
$$\psi'= \sum_{k=0}^\infty  \sum_{j=0}^{p-1} \phi\circ f^{kp+j} $$
in the sense of distributions. Here $BV$ is the space of bounded variation functions in $[a,b]$. This allows us study the regularity of Bikhoff's sums in a far more efective way.  Here $p$ is related with the ergodic decomposition of the absolutely continuous invariant probabilities of  $f$. 

There is a finite number of ergodic absolutely continuous $f$-invariant probabilities  $\mu_\ell$ and densities $\rho_\ell$,  whose (pairwise disjoint) basins of attractions 
$$A_\ell=\{x\in I \ s.t. \ \lim_{N\rightarrow \infty}   \frac{1}{N}  \sum_{k< N}  \theta\circ f^k(x)= \int \theta \ d\mu_\ell, \text{ for every } \theta\in C^0(I)    \}.$$
 covers $m$-almost every point in $I$. Let $$S_\ell= \{x\in I \ s.t. \ \rho_\ell(x) > 0\}\subset A_\ell$$ By Boyarsky and  G\'{o}ra\cite{bg} the set $S_\ell$ is a finite union of intervals up to a zero $m$-measure set. Indeed
 $$A_\ell = \cup_{n=0}^\infty f^{-i}S_\ell$$
 up to  set of zero Lebesgue measure. Define
 $$\Phi_1\colon L^1(m)\rightarrow BV$$
 by 
 $$\Phi_1(\gamma)= \sum_{\ell} \Big( \int_{A_\ell} \gamma \ dm\Big)  \rho_\ell.$$

\begin{mainthm}[Log-Lipchitz continuity]  Let $f\colon I \rightarrow I$ be a piecewise $C^{1+BV}$ expanding map on the interval $I=[a,b]$.  There is  $p\in \mathbb{N}^\star$ such that for every function  $\phi \in L^\infty(I)$ satisfying 
$$ \int \phi \Phi_1(\gamma)  dm=0$$
for every $\gamma \in BV$ we have  $\psi$ is Log-Lipchitz continuous. 
\end{mainthm} 

See Theorem \ref{vvv}  for details. We can ask if the Log-Lipchitz regularity is sharp. Indeed given an ergodic absolutely continuous probability $\mu$ and complex valued functions $\phi_1$ and $\phi_2$  such that 
$$\int \phi_i \ d\nu=0, \ i=1,2$$
we  define 
$$\sigma_\nu(\phi_1,\phi_2)=  \lim_{N \rightarrow \infty}  \int \Big( \frac{\sum_{i=0}^{N-1} \phi_1\circ f^i }{\sqrt{N}}   \Big)\Big( \frac{\sum_{i=0}^{N-1} \overline{\phi}_2 \circ f^i }{\sqrt{N}}   \Big)   \ d\nu$$
whenever this limit exists, and 
$$\sigma^2_\nu(\phi)=\sigma_\nu(\phi,\phi).$$
Note that $\nu$ does not need to be $f$-invariant. 

\begin{mainthm}    Let $f$ be a piecewise $C^{1+BV}$  expanding maps on the interval $I=[0,1]$ and let $\phi$ be a piecewise $C^\beta$ function on  $I$, with $\beta\in (0,1)$, such that 
$$ \int \phi \Phi_1(\gamma)  dm=0$$
for every $\gamma \in BV$.    Then the variance $\sigma_{\mu}(\phi)$ is well-defined and finite  for the Lebesgue measure $m$ on $I$  and for every ergodic absolutely  continuous  $f$-invariant probability $\mu$. Indeed $f$ has only a finite number  of absolutely continuous ergodic $f$-invariant probabilities  $\{\mu_\ell\}_\ell$ and 
$$\sigma_{m}^2 (\phi)= \sum_\ell c_\ell  \sigma_{\mu_\ell}^2(\phi),  $$
where $c_\ell > 0$ and $\sum_\ell c_\ell =1$. If $\sigma_{m}^2 (\phi)=0$ then $\psi$ is a absolutely continuous function and its derivative belongs to $L^2(m)$. 
\end{mainthm} 

See Theorem \ref{cohomo} for details.  For $\sigma_{m}^2 (\phi)> 0$ the regularity of $\psi$ is quite bad.

\begin{mainthm}[Central Limit Theorem for the modulus of continuity]   Let $f$ be a piecewise $C^{2+\beta}$  expanding map on the interval $I=[0,1]$ and let $\phi$ be a piecewise  $C^{\beta}$ function on $I$, with $\beta\in (0,1)$, such that 
$$ \int \phi \Phi_1(\gamma)  dm=0$$
for every $\gamma \in BV$.  Then the variance $\sigma_{\mu}(\phi)$ is well-defined  with respect to every ergodic absolutely  continuous  $f$-invariant probability $\mu$. If $\sigma_{\mu}(\phi) > 0$ then 
 $$\lim_{h\rightarrow  0}  \mu\{ x\in I \colon  \frac{1}{  \sigma_{\mu}(\phi) L \sqrt{-\log |h|}}   \Big( \frac{\psi(x+h)-\psi(x)}{h}\Big) \leq y   \} =\frac{1}{2\pi} \int_{-\infty}^y e^{-x^2} \ dx.   $$
Here
$$L = \Big(\int |Df| \ d\mu\Big)^{-1/2}.$$
In particular $\psi$ is not a Lipschitz function of any measurable subset of positive measure in the support of $\mu$. In particular $\psi$ does not have bounded variation  on the support of $\mu$.
\end{mainthm} 

 See Theorem \ref{cite} for the precise statement. One can ask if $\psi$ is in general a Zygmund function (all Zygmund functions are Log-Lipchitz   continuous). That is not true. See Section \ref{zyg}. 
 
\subsection{Invariant distributions of piecewise expanding maps} \label{maininv} Our results have applications in the study of the  nature of invariant {\it distributions} of a piecewise expanding map. Invariant  finite measures are an obvious example. However there is much more. 

\begin{mainthm} Let $f$ be a piecewise $C^{2+\beta}$  expanding map on the interval $I=[0,1]$, with $\beta \in (0,1)$.  Choose  $\phi \in \mathcal{B}^\beta(C)\cap BV$ such that 
$$\int \phi  d\nu=0$$
for every absolutely continuous $f$-invariant ergodic measure $\nu$. Consider the distribution $\Theta_\phi\in BV^\star$ given by
$$\Theta_\phi(g)=\sigma_m(g,\overline{\phi}).$$
Then $\Theta_\phi$ is $f$-invariant. Moreover $\Theta_\phi$ is a signed measure if and only if  \ $\Theta_\phi=0$. In particular if  $\sigma_m(\phi)> 0$ then $\Theta_\phi$ is  not a signed measure.
\end{mainthm} 

See Theorem \ref{xcxc} for the precise statements.

\section{Regularity under exponential decay of correlations}

\label{expdec}

\subsection{Zygmund and Logarithm Besov spaces} \label{fou}

\label{defsec}

This section contains a series of definitions and results
concerning the regularity properties of functions and 
distributions (in the sense of Schwartz) that we will use in the proof of Theorem \ref{main}.

Let $\calS$ be the Schwartz space of complex-valued 
rapidly decreasing infinitely differentiable functions on $\R^n$
and $\calS^\prime$ the space of continuous linear forms 
 on $\calS$ (temperate distributions). For $\phi\in\calS$ let  $\mathcal{F}(\phi)$ and $\mathcal{F}^{-1}(\phi)$ be  the Fourier transform and its inverse (see  H\"ormander \cite{horm1})

 The Fourier transform
of  $u\in\calS^\prime$, denoted as $\calF u=\hat u$, 
 is defined by $\hat u(\phi)=u(\hat\phi)$. 
The Fourier transform is an isomorphism of $\calS^\prime$ 
(with the weak topology) with inverse given by 
$\calF^{-1}\hat u(\phi)=\hat u(\calF^{-1}\phi)$. If $u\in\calS^\prime$
has compact support then  
$u$ can be extended to the class  of complex-valued 
infinitely differentiable 
functions, denoted as $\cinf$,  
and $\hat u(\xi)=u_x(exp(-ix\cdot\xi))\in\cinf$, where
$u_x$ denotes that $u$ acts on the variable $x$.

Let $\cinf_c(\mathbb{R}^n)$ denote the space of functions in $\cinf(\mathbb{R}^n)$ 
with compact support. Let $\psi_0\in\cinf_c$ be a function
that is radial, non incresing along rays, and   
such that: $\psi_0(x)=1$ for $|x|<1$,
$\psi_0(x)=0$ for $|x|>2$. We define 
$\psi(x)=\psi_0(x)-\psi_0(2x)$ and note that $0\le\psi(x)\le 1$ 
with $\psi(x)=0$
for $|x|<1/2$ and $|x|>2$.
We define $\psi_\ell(x)=\psi(x/2^\ell)$, $\ell\in \mathbb{N}^\star$.  Note that 
\begin{equation}
\text{supp}(\psi_\ell)\subset\{x:2^{\ell-1}\le |x|\le 2^{\ell+1}\}
\label{supp}
\end{equation}
 and 
$\sum_0^N\psi_\ell(x)=\psi_0(x/2^N)\to 1$ as $N\to\infty$, which implies
that the set 
$\{\psi_0,\psi_1,\ldots\}$ yields a partition of unit.
If $u\in\calS^\prime$ then $\psi_\ell(D)u$
 is defined by 
\begin{equation}
\psi_\ell(D)u(x)=\calF^{-1}(\psi_\ell \hat u).\label{psil}
\end{equation}
 Since $\psi_\ell$
has compact support, $\psi_\ell(D)u\in\cinf(\mathbb{R}^n)$. 
For any $s\in\R$ we define the Zygmund class $\Lb^s$ as the 
set of all $u\in\calS^\prime$ with the norm
\[
|u|_s=\sup_{\ell\ge 0}2^{\ell s}\sup|\psi_\ell(D)u|<\infty
\]
The Zygmund class has the following properties (see H\"ormander \cite[Section 8.6]{hormNLH} and  Triebel \cite{triebel}):
\begin{itemize}
\item[(i)] If $s>0$ is not an integer then $u\in\Lb^s$ if, 
and only if,
$u$ is 
a H\"older function with exponent $s$.
\item[(ii)] The Zygmund class $\Lb^1$ consists of all bounded
continuous functions such that 

\[
\sup|u(x)|+\sup_{y\ne 0}\left|\frac{u(x+y)+u(x-y)-2u(x)}{y}\right|<
\infty
\]
and the norm $|u|_1$ is equivalent to the left-hand side. 
There exists analogous characterizations of $\Lb^s$,  
for $s>0$ integer. 
\item[(iii)] If $u\in\Lb^s$ then $\partial_{x_j}u\in \Lb^{s-1}$ and, 
conversely, $u\in\Lb^s$ if $\partial_{x_j}u\in \Lb^{s-1}$, 
$j=1,\ldots,n$.
\item[(iv)] If $u$ is bounded and continuous then $u\in\Lb^0$.
\end{itemize}

Another class of spaces to be considered in this paper is a 
modification of the Zygmund class, the {\it logarithm Besov spaces $B^{s,b}_{\infty,\infty}$}. 
For any $s\in\R$, let $B^{s,b}_{\infty,\infty}$ be 
set of all $u\in\calS^\prime$ such that 
\begin{equation}
\sup_{\ell\ge 0}2^{\ell s}(1+\ell)^b\sup|\psi_\ell(D)u|<\infty
\label{logzyg}
\end{equation}
This  definition of logarithm Besov spaces  can be found in  Cobos,  Dom\'inguez, and   Triebel \cite[Eq. (4.1)]{cobos2016}. Beware that there is another definition of  (sometimes distinct)  logarithm Besov spaces  for $s\geq 0$ in this same reference, using  the modulus of continuity instead of the Fourier transform approach.

 In order to define $\Lambda^s$ and $B^{s,b}_{\infty,\infty}$  on compact  $\cinf$ 
manifolds, it is necessary to consider test functions  
with support in a coordinate domain. A partition of unity can 
be used to decompose test functions with compact support 
into functions with support in coordinate domains. However, in the particular case of  the  $n$-dimensional torus $\Tm^n=\R^n/\Z^n$, it is more convenient to characterize those spaces of distributions using Fourier series instead. 

 The functions and 
distributions on $\Tm^n$ lift to periodic functions and 
periodic distributions
on $\R^n$ (a distribution $u\in\calS^\prime$ is periodic if 
  $u(\phi)=u_x(\phi(x+k))$ for
every $\phi\in\calS$ and $k\in\Z^n$).
 It is convenient to rewrite the definition of $\Lb^s$ 
for periodic distributions (see H\"ormander \cite[Section 7.2]{horm1}).
Let $\Gm\in\cinf_c(\R^n)$ be such that $\sum_{k\in\Z^n}\Gm(x+k)=1$.
It can be shown that any periodic $u\in\calS^\prime$ can 
be written as 
\[
u=\sum_{k\in\Z^n}c_k\erm^{2\pi ix\cdot k}\quad\text{where}\quad
c_k=u_x(\Gm(x)\erm^{-2\pi ix\cdot k}).
\]


\subsection{Proof of  Theorem \ref{main}}

Let $f:M\to M$ be a $\cinf$ diffeomorphism and $\mu$ a smooth 
volume form preserved under $f$. 
Suppose that $M$ is compact
so that  it can be covered 
by a finite number of coordinate patches $U_1,U_2,\ldots$. The
coordinates $x$ on each $U_i$ can be chosen such that 
$\mu=dx_1\wedge dx_2\ldots\wedge dx_n=dx$.
Let $\chi_1,\chi_2\ldots$ be a partition of unit such that 
${\rm supp}\  \chi_i\subset U_i$. If $\phi\in\cinf(M)$ then
$\phi=\phi_1+\phi_2\ldots$ where $\phi_i=\chi_i\phi$ has its support
in $U_i$. If $u$ is a distribution on $M$ then 
$u(\phi)=u(\sum_i\chi_i\phi)=\sum_i  u(\chi_i \phi)$. So, $u$ can be 
decomposed into a sum of distributions $u_i=\chi_i u$, $i=1,2,\ldots$,
such that ${\rm supp}\, u_i\subset U_i$. We say that $u$ 
belongs to some class of regularity (for instance $\in\Lb^s$) if
$u_i$ belong to this class for all $i$. Since 
${\rm supp}\,  u_i\subset U_i\subset\R^n$ 
the analysis of the regularity
of $u_i$ can be made using the tools presented in section 
\ref{defsec}. Since there are finitely many coordinate patches
and the analysis of regularity is similar  in all of them 
we just choose a particular one and neglect the index $i$
associated to it.    

The $u_\ap$ and $u_\om$ distributions restricted to 
a particular coordinate patch are given by
\[
u_\om=\chi(x)\sum^{+\infty}_{j=1}
R(f^{-j}x)\quad\text{and}\quad
u_\ap=-\chi(x)\sum_{j=0}^{+\infty} R(f^{j}x).
\]
The analysis of  the regularity of $u_\ap$ and $u_\om$ are similar, 
so we only consider $u_\ap$. 
Let $\Psi_\ell=\calF^{-1}\psi_\ell\in\calS$. 
Using that  $\calF(\Psi_\ell\ast(\chi u_\ap))=\psi_\ell\calF(\chi u_\ap)$, 
where $\ast$ denotes the convolution,
we obtain that 
\begin{eqnarray*}
\psi_\ell(D)(\chi u_\ap)(x)&=&\calF^{-1}[\psi_\ell(\xi)\calF(\chi u_\ap)](x)=
\Psi_\ell\ast(\chi u_\ap)(x)\\
&=&u_{\ap y}(\Psi_\ell(x-y)\chi(y))=
\sum_{j=0}^\infty\int_{\mathbb{R}^n} \Psi_\ell(x-y)\chi(y)R(f^{j}y)dy.
\end{eqnarray*}
For $\ell=0$, the decay of correlations (\ref{dec1}) and 
the uniform boundness of 
$$|\Psi_0(x-\cdot)\chi(\cdot)|_{{\rm C}^\gm}$$ with respect to $x$
 imply that $\sup|\psi_0(D)(\chi u_\ap)(x)|<\infty$.
For $\ell\ge 1$ we claim that there exists a constant $\Cll{c3}>0$ such that 
\[
\sup|\psi_\ell(D)(\chi u_\ap)(x)|\le\Crr{c3}(\ell+1).
\]

\noindent Indeed, the definition of $\psi_\ell$ for $\ell\ge 1$ implies
\[
\Psi_\ell(x)=2^{\ell n}\Psi(2^{\ell}x)\quad\text{where}\quad
\hat\Psi=\psi
\]
This implies that  
$|\Psi_\ell(x-\cdot)\chi(\cdot)|_{{\rm C}^\gm}\le \Cll{c4}2^{\ell(\gm+n)}$
where $\Crr{c4}>0$ does not depend on $\ell$. So the decay of 
correlations (\ref{dec1}) implies
\[
\left|\int_{\mathbb{R}^n} \Psi_\ell(x-y)\chi(y)R(f^{j}y)dy\right|\le
\Crr{c1}\Crr{c4}\exp(-\Crr{c2}j+\ell(\gm+n)\ln 2)
\]
If $\Cll{c5}=(\gm+n)\ln 2/\Crr{c2}$ then 
\[
\left|\sum_{j\ge \ell \Crr{c5}}\int_{\mathbb{R}^n} \Psi_\ell(x-y)\chi(y)R(f^{j}y)dy
\right|\le \frac{\Crr{c1}\Crr{c4}}{1-\erm^{-\Crr{c2}}}=\Cll{c6}
\]
It remains to estimate 
\begin{align*}
&\left|\sum_{j\le \ell \Crr{c5}}\int_{\mathbb{R}^n}
2^{\ell n}\Psi(2^{\ell}(x-y))\chi(y)R(f^{j}y)dy\right|\\
&=
\left|\sum_{j\le \ell \Crr{c5}}\int_{\mathbb{R}^n}
\Psi(z)\chi(x-2^{-\ell}z)R(f^{j}(x-2^{-\ell}z))dz\right|\\
&\le
|R|_{L^\infty(M)} \sum_{j\le \ell \Crr{c5}}\int_{\mathbb{R}^n}
|\Psi(z)|dz\le \Crr{c7} \ell
\end{align*}
where $\Cll{c7}>0$ does not depend on  $\ell$.
Therefore $$\sup|\psi_\ell(D)(\chi u_\ap)(x)|\le \Crr{c7}\ell+\Crr{c6}\le \Crr{c3}(\ell+1).$$

This completes the proof of the claim.  Consequently
\begin{equation}
\sup_{\ell\ge 0} (1+\ell)^{-1}\sup|\psi_\ell(D)u_\ap|<\infty
\end{equation}
so $u_\ap\in B^{0,-1}_{\infty,\infty}.$


\section{Hyperbolic Linear maps on the Torus}

\label{T2}

In order study  its  regularity 
we use the following.
\begin{proposition} \label{L} For every invertible $n\times n$ hyperbolic matrix $A$ there is $L > 0$ such that
for every  $\ell \in \mathbb{Z}$ and $p\in \mathbb{R}^n$   we have that
\[
2^{\ell}\le|A^j p|\le 2^{\ell+1}
\]
is verified for at most $L$ values of $j\in\Z$.
\end{proposition}
\begin{proof}
Let $E^s$ and $E^u$ be the stable and unstable spaces of $A$. Consider an adapted norm $||\cdot||$ and  $\theta > 1$    such that for every $x\in \mathbb{R}^n$, if we denote $x=x_s+x_u$, with $x_s\in E^s$ and $x_u\in E^u$, we have
\begin{align*}
&||x||=||x_u+x_s|| = ||x_u||+ ||x_s||, \text{ for every $u\in E^u$ and $x_s\in E^s$},\\
&||Ax_u|| \geq \theta ||x_u||, \text{ for every $x_u\in E^u$} ,\\
&||A^{-1}x_s||\geq  \theta ||x_s||, \text{ for every $x_s\in E^s$}.
\end{align*} 
The linearity of $A$ implies that 
it is enough to show that there is $k_0\geq 1$ such that  for every $p\in \mathbb{R}^n$
\begin{equation} \label{disc2}
1\leq || A^{k_0j}  p||\leq 2
\end{equation}
is verified for at most two values $j\in\Z$.

Define 
\begin{align*}
&C^s=\{ x  \ s.t. \   ||x_u||\leq ||x_s||     \},\\
&C^u=\{ x  \ s.t. \   ||x_s||\leq ||x_u||     \}.
\end{align*}
Of course $\mathbb{R}^n= C^s\cup C^u$. Let $k_0$ be such that 
$$ \theta^{k_0} -  \theta^{-k_0}> 6.$$
If $x\in C^u$  then $A^{k_0}x\in C^u$ and 
\begin{align*} 
||A^{k_0}x||&\geq ||A^{k_0}x_u||-  ||A^{k_0}x_s||\geq \theta^{k_0} ||x_u||- \theta^{-k_0}||x_s|| \\
&\geq \Big(  \theta^{k_0} - \theta^{-k_0}\Big) ||x_u||\geq  \Big(  \theta^k - \theta^{-k}\Big) \frac{||x||}{2}  \\
& \geq   3||x||.\label{geq2} \numberthis
\end{align*} 
The $A^{k_0}$-forward invariance of $C^u$ and (\ref{geq2}) implies that
$$\{ j\in \mathbb{Z}\colon \  A^{jk_0}p\in C^u \ and \  1\leq || A^{jk_0}  p||\leq 2     \}$$
contains at most one integer.  By an analogous argument If $x\in C^s$  then $A^{-k_0}x\in C^s$ and 
\begin{equation*} 
||A^{-k_0}x|| \geq  3||x||.
\end{equation*}  
so
$$\{ j\in \mathbb{Z}\colon \  A^{jk_0}p\in C^s \ and \  1\leq || A^{jk_0}  p||\leq 2     \}$$
contains at most one integer.  This completes the proof. 
\end{proof}

\begin{proof}[Proof of Theorem \ref{sigma}] The distributions we are interested in are
of the form $$u(\phi)=\lim_{j\to \pm \infty}Q_j(\phi),$$ where $Q_j(\phi)$ is as in (\ref{Qk}), taking $\rho_0=0$ in (\ref{coc}). We will prove the theorem for  $u_\alpha$. The proof of the regularity of $u_\omega$ is analogous.
Using the notation of Section \ref{fou} 
\[
u_\alpha(\phi)= \sum_{k\in\Z^n}c_k\int \erm^{2\pi ix\cdot k}\phi(x)dx
\]
for every $\phi\in\calS$, where
\begin{equation}
c_k=\lim_{ j\to -\infty}
\int_{\R^n}\rho_j(x)\Gm(x)\erm^{-2\pi ix\cdot k}\ dm(x)=
\lim_{j\to -\infty}
\int_{\Tm^n}\rho_j(x)\erm^{-2\pi ix\cdot k} \ dm(x),
\label{ck}
\end{equation}
where $\rho_j$ is given in equation (\ref{coc}).

Let $\dt_z$ be the $\delta$-Dirac distribution  with 
support at the point $z\in\R^n$.
Using that $\hat\dt_z(\xi)=\erm^{- i z\cdot\xi}$
 and $\calF\erm^{ i x\cdot\xi}=(2\pi)^n\dt_\xi$
we obtain 
the Fourier transform of $u$:
\[
\hat u=(2\pi)^n\sum_{k\in\Z^n}c_k\dt_{2\pi k}.   
\]
This and  (\ref{psil}) imply 
\begin{equation}
u_\ell(x)=\psi_\ell(D)u(x)=\calF^{-1}(\psi_\ell(\xi)\hat u)=
\sum_{k\in\Z^n}c_k\psi_\ell(2\pi k)\erm^{-2\pi ix\cdot k},\label{uj}
\end{equation}
and 
\[
|u|_s=\sup_{\ell\ge 0}2^{s\ell}\sup|u_\ell(x)|
\]
We provide the proof for $u_\ap$. 
Let $R\in \Lambda^\beta$.
We have $$R(x)=\sum_{p\in \mathbb{Z}^n} b_p\exp(2\pi i p\cdot x).$$
Denote 
\begin{align*} 
u_{\ap,p}(x)&=-\sum_{j=0}^\infty\exp(2\pi i p\cdot f^{j}x)=
-\sum_{j=0}^\infty\exp(2\pi i x\cdot (M^\star)^jp).
\end{align*} 
Consequenlty
$$u_\ap= \sum_{p\in \mathbb{Z}^n} b_p u_{\ap,p}.$$
From  (\ref{ck}) 
\[
c_k=-\sum_{j=0}^\infty\sum_{p\in\Z^n} b_p\int_{\Tm^n}
\erm^{-2\pi ix\cdot ((M^\star)^jp-k)}dx=
-\sum_{p\in\Z^n}\sum_{j=0}^\infty b_p\dt_{k,(M^\star)^j p}
\]
for every $k\in \mathbb{Z}^n$, where $\dt_{i,j}=1$ if $i=j$, otherwise 
$\dt_{i,j}=0$. 
From (\ref{uj}) and Proposition \ref{L}
\begin{eqnarray*}
|u_{\ell}(x)|&=&\left|
\sum_{k\in\Z^n}\sum_{p\in\Z^n}\sum_{j=0}^\infty 
b_p\dt_{k,(M^\star)^j p}
\psi_\ell(2\pi k)\erm^{-2\pi ix\cdot k}\right|\\
&=&
\left|
\sum_{p\in\Z^n}b_p\sum_{j=0}^\infty 
\psi_\ell(2\pi(M^\star)^j p)\erm^{-2\pi ix\cdot (M^\star)^j p}\right|
\le L \sum_{p\in\Z^n}|b_p|.
\end{eqnarray*}
Note that  $\sum_{p\in\Z^n}|b_p|$ converges because $R\in{\rm C}^\beta$ with $\beta>n/2$ (see for instance Grafakos \cite[Theorem 3.2.16] {grafakos}). 
\end{proof}

\begin{rmk} 
Notice that 
\[
u=u_\omega-u_\alpha= \sum_{j\in\Z}\exp(2\pi i p\cdot f^{j}x)\in\Lb^0
\]
is a weak solution of the  equation
$u\circ f-u=0$, that is, $u$ is a $f$-invariant distribution.
\end{rmk}

\subsection{Regularity on invariant foliations} Note that typically $u_\alpha$ and $u_\beta$ {\it  are  not } functions. Indeed it is well known  since Liv\v{s}ic \cite{livsic} that for a residual subset of  functions $R\in C^\alpha$ the cohomological equation $R=u\circ f - u$ do not have a continuous solution $u$ when $f$ is a  Anosov diffeomorphism. However,  the  distributions  $u_\om$ and $u_\ap$ 
 have directional
derivatives with  different  regularity properties. Let $E^s$ and $E^u$ be the stable and unstable directions of $M$. 
The weak derivative
of $u_\ap$ with respect to $s\in E^s$ is
\begin{align*}
D_s u_\ap(\phi)&=-u_\ap(D_s \phi)=
\lim_{k\rightarrow +\infty}  - \int_{\mathbb{T}^n} \sum_{j=0}^k R \circ f^j \cdot D_s \phi \ dm\\
&= \lim_{k\rightarrow +\infty}  -    \int_{\mathbb{T}^n} \sum_{j=0}^k DR(f^j(x))\cdot Df^j(x)\cdot s\   \phi(x) \ dm(x)
\end{align*} 
If $\lambda \in (0,1)$ satisfies $|Df^j(x)\cdot s|\leq C\lambda^j  |s|$, for every $s\in E^s$ then
$$
| DR(f^j(x))\cdot Df^j(x)\cdot s|= C \lambda^j |DR(f^j(x)) ||s|
$$
we obtain that $$ \sum_{j=0}^\infty DR(f^j(x))\cdot Df^j(x)\cdot s$$
converges uniformly in $x$ and therefore $D_su_\ap$ is a 
continuous function. So $u_\ap\in \Lb^0$ is differentiable
in the stable direction and its lack of regularity is related
to the unstable direction (the ``Wave-front set'' of $u_\ap$ is 
in the unstable direction, see \cite{horm1} chapter VIII for details).
The same sort of analysis shows that $D_w u_\om$ is 
continuous, for every $w\in E^u$. So, $u_\om$ is differentiable
in the unstable direction and its lack of regularity is related 
to the stable direction. 

\begin{rmk} Note that we are interested in the isotropic regularity of the Birkhoff sums.   The logarithm Besov spaces $B^{s,b}_{\infty,\infty}$ are isotropic spaces. All directions are treated in the same way. In the  general case of an  (nonlinear) Anosov diffeomorphisms on a compact manifold,  
the stable and unstable are typically just  H\"older invariant foliations, so it  is a more difficult setting, and it deserves
further research.  Anisotropic Banach spaces will certainly be quite useful here, since much more it is known on the regularity of the Koopman operator  for many anisotropic Banach spaces in the literature, and consequently the "anisotropic"  regularity of the solutions of the cohomological equation. See  Baladi and Tsujii \cite{bt2}, Blank, Keller and  Liverani \cite{bkl}, Gou\"{e}zel and  Liverani \cite{gl2}, and Baladi \cite{bbook}. 
\end{rmk} 

\subsection{Birkhoff sums as derivatives of infinitesimal conjugacies} The interest on Birkhoff sums as distributions can be motived by the following problem. Let $F_t$ be a $C^\beta$-smooth family of $C^\beta$-Anosov diffeomorphisms on $\mathbb{T}^2$, with $\beta > 2$, and  such that $F_0$ is the Arnold's cat map. Since Anosov maps are structurally stable there is a family of homeomorphisms $H_t$ such that 
$$H_t\circ F_0 = F_t\circ H_t$$
with $H_0(x)=x$. It is not difficult to see that for each $x\in \mathbb{T}^2$ the map
$$t\mapsto H_t(x)$$
is smooth. If $W=\partial F_t|_{t=0}$ and $\alpha=\partial_t H_t|_{t=0}$ then
$$W =\alpha\circ F_0 -  DF_0\cdot \alpha.$$
From now on it is more convenient to consider $W$ and $\alpha$ as $\mathbb{Z}^2$-periodic functions on $\mathbb{R}^2$.  Let $\pi_s\colon \mathbb{R}^2\rightarrow E^s$ and $\pi_u\colon \mathbb{R}^2\rightarrow E^u$ be linear projections on the stable and unstable directions of $F_0$ with $\pi_s(x)+\pi_u(x)=x$. Let $\lambda_s$ and $\lambda_u$ be the stable and unstable eigenvalues of $F_0$ (note that $\lambda_s\lambda_u=1$) and by $v_s$ and $v_u$ the respective eigenvectors with $|v_s|=|v_u|=1$.  Using $B=(v_s,v_u)$ as a base, and $v=xv_s+yv_u=(x,y)_B$,   we can write 
$$\pi_s\circ W = \alpha_s \circ F_0 - DF_0\cdot \alpha_s,$$
$$\pi_u\circ W = \alpha_u \circ F_0 - DF_0\cdot \alpha_u,$$
that implies 
$$\alpha_s(x,y)=     -  \sum_{k=0}^\infty  DF_0^{k+1}\cdot  \pi_s\circ W(F_0^{-k}(v))=-  \sum_{k=0}^\infty  \lambda_s^k \pi_s\circ W(\lambda_s^{-k}x,\lambda_u^{-k}y),$$
$$\alpha_u(x,y)=     -  \sum_{k=0}^\infty  DF_0^{-(k+1)}\cdot  \pi_u\circ W(F_0^{k}(v))= -  \sum_{k=0}^\infty \lambda_u^{-k}   \pi_u\circ W(\lambda_s^kx,\lambda_u^k y),$$
so  $\alpha=\alpha_s+\alpha_u$. We call $\alpha$ the {\it infinitesimal deformation} associated to $V$ and $F_0$. If we formally derive $\alpha$ we get

$$\partial_s \alpha(x,y)=  -  \sum_{k=0}^\infty \pi_s\circ \partial_s W(F_0^{-k}(x,y))  -  \sum_{k=0}^\infty \Big(\frac{\lambda_s}{\lambda_u}\Big)^k  \pi_u\circ \partial_s W(F_0^{-k}(x,y)),$$

$$\partial_u \alpha(x,y)=  -  \sum_{k=0}^\infty \pi_u\circ \partial_u W(F_0^{k}(x,y))  -  \sum_{k=0}^\infty \Big(\frac{\lambda_s}{\lambda_u}\Big)^k  \pi_s\circ \partial_u W(F_0^{k}(x,y)),$$

The first term in both expressions is a  Birkhoff sum (in distinct time directions). The second term are continuous  functions since $|\lambda_s/\lambda_u| < 1$. So the regularity of $\alpha$ depends on the regularity of Birkhoff sums. 
Theorem  \ref{sigma} implies $\partial_s \alpha, \partial_u \alpha\in \Lambda^0$, so $\alpha \in \Lambda^1$, that is, $\alpha$ is a  Zygmund function.

It is an intriguing observation,  up to this point  limited to simple linear Anosov diffeomorphisms. One may ask if we can study the regularity of infinitesimal deformations of nonlinear Anosov diffeomorphisms using such methods. This poses new difficulties since the stable and unstable foliations are typically far less regular.  

A similar  study of deformations of one-dimensional piecewise expanding maps allows us to give a far more complete picture. See G.R. and S.  \cite{dois}  and previous results by Baladi and S. \cite{bs0} \cite{smooth} \cite{alternative}.

\section{Piecewise expanding maps: Primitives of Birkhoff  sums}

We define $I=[a,b]$. Let $C=\{c_0,c_1, \dots, c_n\}$, with $c_0=a$, $c_n=b$, and  $c_i< c_{i+1}$, for every $i< n$.  Given $n\in \mathbb{N}$ and $\beta\in [0,1)\cup\{BV\}$, let $\mathcal{B}^{n+\beta}(C)$ be the space of all functions 
$$v\colon \cup_{i< n}  (c_i,c_{i+1})\rightarrow \mathbb{C}$$
such that 
\begin{itemize}
\item For each $i< n$, $v$ can be extended to a function $v_i\colon [c_i,c_{i+1}]\rightarrow \mathbb{C}$ which  is $n-1$ times differentiable and $\partial^{n-1}v_i$ is absolutely continuous and its derivative is continuous for $\beta=0$, it is $\beta$-H\"older, if $\beta\in (0,1)$, and has bounded variation of $\beta=BV$.
\end{itemize} 
Let $$\hat{I}=\{a^+,b^-\}\cup \{x^+,x^-\colon x\in (a,b)\}.$$
Every $v\in \mathcal{B}^{n+\beta}(C)$ induces a function $v\colon \hat{I}\rightarrow \mathbb{C}$  defined by
$$v(x^\star)=\lim_{z\rightarrow x^\star} v(z),$$
where $x\in I$, $\star\in \{+,-\}$ and $x^\star\in \hat{I}$.

We will denote $\mathcal{B}^{n+0}(C)$ by $\mathcal{B}^{n}(C)$. Let $\mathcal{B}^{n+\beta}_{exp}(C)$, with $n\geq 1$, be the set of all $f\in \mathcal{B}^{n+\beta}(C)$ such that 
\begin{itemize}
\item $f$ is monotone on each interval $(c_i,c_{i+1})$, $i< n$.
\item For every $i$ we have $f_i[c_i,c_{i+1}]\subset I$.
\item   There is $\theta > 1$ such that  
$$\min_{i< n} \inf_{x\in [c_i,c_{i+1}]}  |Df_i(x)|\geq \theta.$$
\end{itemize} 
Note that $f^k\in \mathcal{B}^{n+\beta}_{exp}(C_k)$, for some set $C_k$ and we can indeed define an extension $f^k\colon \hat{I}\rightarrow  \hat{I}$ using lateral limits. Moreover if $v\in \mathcal{B}^{n+\beta}(C)$ then $v\circ f^k \in  \mathcal{B}^{n+\beta}(C_k)$. Let $f\in \mathcal{B}^{1+BV}(C)$. Let $L$ be the {\it transfer operator}  of $f$ associated with the Lebesgue measure $m$ on $I$. Then Lasota and Yorke \cite{ly} proved that \\

\begin{itemize}
\item{\bf (Lasota-Yorke inequality in BV)}  There is $\Cll{b}$ and $\Cll[c]{c}$ such that  
\begin{equation}\label{ly} |L \gamma|_{BV}\leq    \Crr{c} |\gamma|_{BV}+ \Crr{b}|\gamma|_{L^1}. \end{equation}
\end{itemize} 

This  implies
\begin{equation}\label{lyy} |L^i \gamma|_{BV}\leq  \Cll{ff} \Crr{c}^i |\gamma|_{BV}+ \Cll{0}|\gamma|_{L^1}. \end{equation}
for every $i$. Let
$$\Lambda = \{ \lambda\in \mathbb{S}^1\colon \ \lambda \in \sigma(L)  \},$$
where $ \sigma(L)$ is the spectrum of $L$ in $BV$. Lasota-Yorke inequality implies  that  $\Lambda$ is finite, $1\in \Lambda$  and we can write 
\begin{equation}\label{ii}  L = \sum_{\lambda\in \Lambda}   \lambda \Phi_\lambda   +    K\end{equation}
where $\Phi_\lambda^2=\Phi_\lambda$, $\Phi_\lambda \Phi_{\lambda'}=0$ if $j\neq j'$ and $K\Phi_\lambda=\Phi_{\lambda}K=0$. Moreover
\begin{itemize}
\item[i.]  $\Phi_\lambda$ is a finite rank  operator.
\item[ii.] $K$ is a bounded operator in $BV$ whose  spectral radius  is smaller than one, that is, there is $\Cll[c]{cpf}  \in (0,1)$ and $\Cll{cd}$ such that 
$$|K^j(\phi)|_{BV}\leq \Crr{cd}\Crr{cpf}^j |\phi|_{BV}.$$
\item[iii.]  there is $p=p(f)\in \mathbb{N}^\star$ such that for every $\lambda\in \Lambda$ we have $\lambda^p=1$. \\
\end{itemize} 

We will use Lasota-Yorke result many times along this work. Let $\Cll[c]{cmin}= sup_{x\in I} |Df(x)|^{-1}.$

\begin{lemma}\label{uu}   There is $\Crr{0}$, that depends only the constants in the Lasota-Yorke inequality, such that 
$$|\Phi_\lambda(\gamma)|_{BV}\leq \Crr{0} |\gamma|_{L^1(m)}$$
for every $\gamma\in BV$. 
\end{lemma} 
\begin{proof}  It follows from  (\ref{ii}) that 
$$\lim_n \frac{1}{np}\sum_{i=0}^{np-1} \frac{1}{\lambda^i} L^{i}\gamma = \Phi_\lambda(\gamma)$$
in $BV$. So  (\ref{lyy}) implies
$$|\Phi_\lambda(\gamma)|_{BV}\leq  \Crr{0} |\gamma|_{L^1}.$$

\end{proof}

\begin{lemma}\label{kk} Let $p=p(f).$ For every $\gamma\in BV$ and $\phi \in L^1(I)$ such that  $$\int \phi \Phi_1(\gamma) \ dm=0.$$
we have
\begin{align}
\int \phi  \sum_{j=pj_1}^{pj_2-1} K^{j}(\gamma) \ dm &=\sum_{k=j_1}^{j_2-1}  \int \phi \sum_{j=0}^{p-1} L^{kp+j}\gamma \ dm \\
&=  \int  \gamma \cdot \Big( \sum_{k=j_1}^{j_2-1}  \sum_{j=0}^{p-1} \phi\circ f^{kp+j} \Big)\ dm.\nonumber
\end{align}
for every $j_1,j_2\in \mathbb{N}\cup \{+\infty\}$, $j_1\leq j_2$.
\end{lemma} 
\begin{proof} For every $\gamma \in BV$ we have that 
$$\int \phi  \sum_{j=pj_1}^{pj_2-1} K^{j}(\gamma) \ dm$$
converges, since the spectral radius os $K$ is smaller than one in $BV$. Since 
$$ \sum_{j=0}^{p-1}  \lambda^{j}=0$$ 
for $\lambda\in \Lambda\setminus \{1\}$  we have 
\begin{align*}
\int \phi  \sum_{j=pj_1}^{pj_2-1} K^{j}(\gamma) \ dm&=\sum_{k=j_1}^{j_2-1}  \int \phi \Big( \sum_{j=0}^{p-1} K^{kp+j}(\gamma) \Big) \ dm  \\
&=\sum_{k=j_1}^{j_2-1} \int \phi \Big( \sum_{j=0}^{p-1}  \Phi_1(\gamma) + \sum_{j=0}^{p-1} K^{kp+j}(\gamma) \Big) \ dm  \\
&=\sum_{k=j_1}^{j_2-1}   \int \phi \Big(\sum_{\lambda\in \Lambda}  \lambda^{kp} \sum_{j=0}^{p-1}  \lambda^{j} \Phi_\lambda(\gamma) + \sum_{j=0}^{p-1} K^{kp+j}(\gamma) \Big) \ dm \\
&=\sum_{k=j_1}^{j_2-1}   \int \phi  \Big( \sum_{j=0}^{p-1} \sum_{\lambda\in \Lambda} \lambda^{kp+j} \Phi_\lambda(\gamma) + \sum_{j=0}^{p-1} K^{kp+j}(\gamma) \Big) \ dm \\
&=\sum_{k=j_1}^{j_2-1}  \int \phi \sum_{j=0}^{p-1} L^{kp+j}\gamma \ dm \\
&=\sum_{k=j_1}^{j_2-1}  \int  \gamma \cdot \Big( \sum_{j=0}^{p-1} \phi\circ f^{kp+j} \Big)\ dm\\
&=  \int  \gamma \cdot \Big( \sum_{k=j_1}^{j_2-1}  \sum_{j=0}^{p-1} \phi\circ f^{kp+j} \Big)\ dm.
\end{align*}

\end{proof}

\subsection{Log-Lipschitz regularity} 

\begin{theorem} \label{vvv} Let $f\colon I \rightarrow I$ be a piecewise $C^{1+BV}$ expanding map on the interval $I=[a,b]$. 
Let $p$ be  a multiplier of $p(f).$  There are  $\Cll{cdd}$, $\Cll{e}$ and $\Cll{f}$ with the following property. Let $\phi\colon I\rightarrow \mathbb{C}$ be a function in $L^\infty(I)$ such that 
$$\int \phi \cdot \Phi_1(\gamma) \ dm=0$$
for every $\gamma \in BV$.  Then 
\begin{itemize}
\item[A.] For every $\gamma\in BV$  there is  $i_0 \in p\mathbb{N}$ such that 
$$\frac{\ln |\gamma|_{BV} - \ln  |\gamma|_{L^1(m)}}{-\ln \Crr{c}} - p(f) \leq i_0\leq \frac{\ln |\gamma|_{BV} - \ln  |\gamma|_{L^1(m)}}{-\ln \Crr{c}} + p(f)$$
and
\begin{align*}
&\big| \int  \gamma \cdot \Big( \sum_{k=i_0/p}^{\infty} \sum_{j=0}^{p-1} \phi\circ f^{kp+j} \Big)\ dm \Big|\\
&\leq \Crr{cdd} |\phi|_{L^\infty(m)}    |\gamma|_{L^1}.
\end{align*}
\item[B.] For every $\gamma\in BV$ 
\begin{align}\label{eess}  & \sum_{k=0}^\infty  \Big|   \int  \gamma \cdot \Big(  \sum_{j=0}^{p-1} \phi\circ f^{kp+j} \Big)\ dm \Big| \nonumber \\
& \leq    \Crr{e}  |\phi|_{L^\infty(m)} ( ( \ln |\gamma|_{BV}-\ln |\gamma|_{L^1(m)})   + \Crr{f})  |\gamma|_{L^1(m)}.
\end{align} 
\item[C.] We have that 
$$\psi(x)=   \int  1_{[a,x]}\cdot \Big( \sum_{k=0}^\infty  \sum_{j=0}^{p-1} \phi\circ f^{kp+j} \Big)\ dm,$$
is well defined for every $x\in I$ and 
\begin{equation} \label{eesss} |\psi(x)-\psi(y)|\leq  \Crr{e}  |x-y| |\ln |x-y||+ (\Crr{f}+\ln (2+|I|))|x-y|.\end{equation}  
  The constants $\Crr{e}$ and $\Crr{f}$ depend only on $\Crr{b}$,   $m(I)$ and $|\phi|_{L^\infty}$. 
  \item[D.] Define
 \begin{equation}\label{psin} \psi_n(x)=   \int  1_{[a,x]}\cdot \Big( \sum_{k=0}^n  \sum_{j=0}^{p-1} \phi\circ f^{kp+j} \Big)\ dm.\end{equation}
 then for every $\beta\in (0,1)$ we have
 $$\lim_n |\psi_n-\psi|_{C^\beta(I)}=0.$$
 \item[E.] For every $\gamma\in BV$  we have 
 $$\int \gamma \ d\psi=   \int  \gamma  \Big( \sum_{k=0}^\infty  \sum_{j=0}^{p-1} \phi\circ f^{kp+j} \Big)\ dm,$$
 where the left-hand side is a (central) Young integral. Moreover for $\gamma\in C^\infty(I)$ 
 \begin{equation}\label{distb}  -\int \psi D\gamma \ dm =  -\gamma(b)\psi(b)+  \int  \gamma  \Big( \sum_{k=0}^\infty  \sum_{j=0}^{p-1} \phi\circ f^{kp+j} \Big)\ dm\end{equation}
 so we have 
 $$D\psi =-\psi(b)\delta_b +  \sum_{k=0}^\infty  \sum_{j=0}^{p-1} \phi\circ f^{kp+j} $$
 in the sense of distributions. 
  \end{itemize}
\end{theorem}
\begin{proof}  By  Lemma \ref{kk} 
\begin{align*}
&\int  \gamma \cdot \Big( \sum_{j=0}^{p-1} \phi\circ f^{kp+j} \Big)\ dm \\
&=   \int \phi \sum_{j=0}^{p-1} K^{kp+j} ( \gamma )\ dm \\
&=   \int \phi \sum_{j=0}^{p-1} L^{kp+j} ( \gamma )\  dm \\
&=   \int \phi \sum_{j=0}^{p-1} L^{kp+j} ( \gamma- \Phi_1(\gamma))\  dm
\end{align*}
for every $k\in \mathbb{N}$. Due Lemma \ref{uu} we have
\begin{align*}&|L^{i} \big(\gamma- \Phi_1(\gamma)\big)|_{L^1(m)}\\
&\leq |\gamma- \Phi_1(\gamma)|_{L^1(m)}\leq \Cll{we}  |\gamma|_{L^1(m)}\end{align*} 
for every $i$. Let $i_0=i_0(\gamma)$ be the largest  $i_0 \in p\mathbb{N}$ such that 
$$ \Crr{c}^{i_0}|\gamma|_{BV} \leq    |\gamma|_{L^1(m)}.$$
Then 
\begin{equation}\label{98}  \frac{\ln |\gamma|_{BV} - \ln  |\gamma|_{L^1(m)}}{-\ln \Crr{c}} - p(f)  \leq i_0\leq \frac{\ln |\gamma|_{BV} - \ln  |\gamma|_{L^1(m)}}{-\ln \Crr{c}} + p,\end{equation} 
so
\begin{align*} &\sum_{i=0}^{i_0-1}  |L^{i}( \gamma- \Phi_1(\gamma))|_{L^1(m)}\\
& \leq \Cll{ty}  i_0 |\gamma- \Phi_1(\gamma)|_{L^1(m)}\\
&\leq  \Cll{er}  (  \frac{\ln |\gamma|_{BV}-\ln |\gamma|_{L^1(m)}}{-\ln \Crr{c}}   + p)  |\gamma|_{L^1(m)}.
\end{align*} 
and
\begin{align*}
&\big| \int  \gamma \cdot \Big( \sum_{k=0}^{i_0/p-1} \sum_{j=0}^{p-1} \phi\circ f^{kp+j} \Big)\ dm \Big|\\
&= \Big|   \int \phi \sum_{j=0}^{i_0/p-1}  \sum_{j=0}^{p-1} L^{kp+j} ( \gamma- \Phi_1(\gamma))\  dm \Big| \\
&\leq |  \phi|_{L^\infty(m)}  \sum_{j=0}^{i_0/p-1} \sum_{j=0}^{p-1}  | L^{kp+j} ( \gamma- \Phi_1(\gamma))|_{L^1(m)} \\
&\leq  \Crr{er} |\phi|_{L^\infty(m)}  (  \frac{\ln |\gamma|_{BV}-\ln |\gamma|_{L^1(m)}}{-\ln \Crr{c}}   + p)  |\gamma|_{L^1(m)}.  \numberthis \label{a11}
\end{align*}
By  (\ref{lyy})
\begin{align*}& |L^{i_0}( \gamma- \Phi_1(\gamma))|_{BV}\leq  \Crr{ff} \Crr{c}^{i_0} |\gamma- \Phi_1(\gamma)|_{BV}+ \Crr{0}|\gamma- \Phi_1(\gamma)|_{L^1}\\
&\leq  \Cll{fff} |\gamma|_{L^1(m)}.
\end{align*} 
so for every $n\geq i_0/p$ and $m\in \mathbb{N}\cup\{\infty\}$
\begin{align*}
&\big| \int  \gamma \cdot \Big( \sum_{k=n}^{m} \sum_{j=0}^{p-1} \phi\circ f^{kp+j} \Big)\ dm \Big|\\
&= \Big|   \int \phi \sum_{k=n}^{m}  \sum_{j=0}^{p-1} L^{kp+j} ( \gamma- \Phi_1(\gamma))\  dm \Big| \\
&= \Big|   \int \phi \sum_{k=n-i_0/p}^{m-i_0/p}  \sum_{j=0}^{p-1} L^{kp+j} L^{i_0}( \gamma- \Phi_1(\gamma))\  dm \Big| \\
&= \Big|   \int \phi \sum_{k=n-i_0/p}^{m-i_0/p}  \sum_{j=0}^{p-1} K^{kp+j} L^{i_0}( \gamma- \Phi_1(\gamma))\  dm \Big| \\
&\leq  \Crr{cd} |\phi|_{L^\infty(m)}         | L^{i_0}( \gamma- \Phi_1(\gamma))|_{BV}\Crr{cpf}^{n-i_0/p} \sum_{j=0}^{\infty} \sum_{j=0}^{p-1} \Crr{cpf}^{kp+j}\\
&\leq \Crr{cdd} |\phi|_{L^\infty(m)} \Crr{cpf}^{n-i_0/p}   |\gamma|_{L^1}.  \numberthis \label{a12}
\end{align*}
Estimates (\ref{a11}) and (\ref{a12}) imply A.  and $B.$ If we choose $\gamma=1_{[a,x]}$ then $B.$ implies $C$.

Let $\psi_n$ be as in (\ref{psin}). Then 

$$\psi(x)-\psi_n(x)=    \int  1_{[a,x]}\cdot \Big( \sum_{k=n+1}^\infty  \sum_{j=0}^{p-1} \phi\circ f^{kp+j} \Big)\ dm. $$
and 
$$|\psi -\psi_n|_{C^\beta(I)}= \sup_{\delta < |I|} \sup_{\substack{x \in I \\ x+\delta \in I}}  \frac{1}{\delta^\beta} \Big|  \int  1_{[x,x+\delta]}\cdot \Big( \sum_{k=n+1}^\infty  \sum_{j=0}^{p-1} \phi\circ f^{kp+j} \Big)\ dm \Big|. $$

Note that $|1_{[x,x+\delta]}|_{BV}= \delta+ 2 \leq 2+|I|$ and $|1_{[x,x+\delta]}|_{L^1(m)}=\delta$. If 

\begin{equation}\label{yyy} np \geq i_0(1_{[x,x+\delta]}),\end{equation} 
note that (\ref{98}) gives us 
$$\delta\leq \Cll{new} \Crr{c}^{i_0(1_{[x,x+\delta]})} $$
for some $\Crr{new}$. Let $\Cll[c]{max}=\max \{\Crr{c}^{1-\beta},\Crr{cpf}^{1/p} \}$. Then  (\ref{a12}) implies 

\begin{align*}  &\frac{1}{\delta^\beta} \Big|  \int  1_{[x,x+\delta]}\cdot \Big( \sum_{k=n+1}^\infty  \sum_{j=0}^{p-1} \phi\circ f^{kp+j} \Big)\ dm \Big| \\
& \leq  \Crr{cdd} |\phi|_{L^\infty(m)}\Crr{cpf}^{n-i_0/p}  \delta^{1-\beta}\\
& \leq \Cll{new2} |\phi|_{L^\infty(m)} \Crr{max}^{np}. \end{align*} 
On the other hand, if (\ref{yyy}) does not hold, then (\ref{a11}) and (\ref{a12}) imply  
\begin{align*}  &\frac{1}{\delta^\beta} \Big|  \int  1_{[x,x+\delta]}\cdot \Big( \sum_{k=n+1}^\infty  \sum_{j=0}^{p-1} \phi\circ f^{kp+j} \Big)\ dm \Big| \\
& \leq   |\phi|_{L^\infty(m)} \delta^{1-\beta}  (  \Crr{er} (  \frac{\ln (2+|I|) -\ln \delta }{-\ln \Crr{c}}   + p-np) +  \Crr{cdd} )
\end{align*} 
and 
\begin{equation}\label{cond1}  \frac{\ln (2+|I|) -\ln \delta }{-\ln \Crr{c}}   + p-np \geq 0.\end{equation} 
which is equivalent to 
$$\delta\leq (2+|I|) \Crr{c}^{p(n-1)}.$$
So if we define $h_n\colon (0,|I|]\rightarrow \mathbb{R}_+^\star$ as 
$$h_n(\delta)=\begin{cases}   \delta^{1-\beta}  (  \Crr{er} (  \frac{\ln (2+|I|) -\ln \delta }{-\ln \Crr{c}}   + p-np) +  \Crr{cdd})+ \Crr{new2}  \Crr{max}^{np}, &     if \ \delta\leq (2+|I|) \Crr{c}^{p(n-1)},\\
 \Crr{new2} \Crr{max}^{np}, \ & \  otherwise.
\end{cases} 
$$
then 
$$|\psi -\psi_n|_{C^\beta(I)}\leq  \sup_{\delta < |I|}  h_n(\delta).$$
Consequently it  is easy to see that $$|\psi -\psi_n|_{C^\beta(I)}\leq \Cll{new3} \Crr{max}^{pn}.$$
This proves D.  In particular
$$\lim_n |\psi -\psi_n|_{BV_{1/\beta}}=0,$$
for every $\beta\in (0,1)$, so Love-Young inequality (see  Lyons, Caruana  and L\'{e}vy \cite[Theorem 1.16]{lyons}) implies 
$$\lim_n \int  \gamma  \ d \psi_n = \int  \gamma    \ d\psi,$$
where all integrals are central Young integrals. On the other hand, since $\psi_n$ is absolutely continuous and $D\psi_n \in BV$ we have
$$\int  \gamma  \ d \psi_n =    \int  \gamma  \Big( \sum_{k=0}^n \sum_{j=0}^{p-1} \phi\circ f^{kp+j} \Big)\ dm, $$
and consequently 
$$ \int  \gamma    \ d\psi=\lim_n \int  \gamma  \ d \psi_n =   \int  \gamma  \Big( \sum_{k=0}^\infty  \sum_{j=0}^{p-1} \phi\circ f^{kp+j} \Big)\ dm.$$
If $\gamma\in C^\infty(I)$ then (\ref{distb}) follows from integration by parts for  Young integrals (see Hildebrandt \cite{Hildebrandt}). This concludes the proof of E.
\end{proof}
\begin{rmk} Using Keller \cite{keller} generalised bounded variations spaces one could prove Theorem \ref{vvv} assuming $\gamma\in BV_q$, with $q\geq 1$. 
\end{rmk} 

 Define
$$\tilde{\psi}_n(x)=\int   1_{[a,x]}\Big( \sum_{k=0}^{n}\phi\circ f^{k}\Big)\  dm.$$
If $p(f)\neq 1$ then $\lim_n \tilde{\psi}_n(x)$ may not exist. But their Ces\`aro mean does converge.
For every $\gamma \in BV$   denote
$$T_u(\gamma)=\frac{1}{u}\sum_{n=0}^{u-1}   \int \gamma \sum_{k=0}^{n} \phi(f^k(u)) \ du.$$

\begin{theorem}\label{conv2} Let $\phi\colon I\rightarrow \mathbb{C}$ be a function in $L^\infty(I)$ such that 
$$\int \phi \Phi_1(\gamma)=0$$
for every $\gamma \in BV$. Then for every $\gamma\in BV$ 
$$\lim_u T_u(\gamma)=    \int \gamma \sum_{n=0}^\infty \sum_{j=0}^{p-1} \phi\circ f^{pn+j} \ dm +  \int \phi \Big( \sum_{\lambda\in \Lambda\setminus\{1\}}  \frac{1}{1-\lambda}  \Phi_\lambda(\gamma)    \Big)\ dm. $$
In particular if 
$$\hat{\psi}_u(x)= \frac{1}{u} \sum_{n=0}^{u-1} \tilde{\psi}_n(x).$$
then
$$\lim_u \hat{\psi}_u(x)=   \int 1_{[a,x]} \sum_{n=0}^\infty \sum_{j=0}^{p-1} \phi\circ f^{pn+j} \ dm + G(x),
$$
where 
$$G(x)= \int \phi \Big( \sum_{\lambda\in \Lambda\setminus\{1\}}  \frac{1}{1-\lambda}  \Phi_\lambda(1_{[a,x]})    \Big)\ dm $$
is a Lipchitz function.
\end{theorem}
\begin{proof} Note that for every $\gamma \in BV$  (the manipulations with eigenvalues   are  as  those  in Broise \cite{broise})
\begin{align*} 
&T_n(\gamma)=\frac{1}{u}\sum_{n=0}^{u-1}   \int \gamma \sum_{k=0}^{n} \phi(f^k(u)) \ du \\
&=  \int  \phi  \frac{1}{u}\sum_{n=0}^{u-1} \sum_{k=0}^{n}L^k\gamma \ dm\\
&=  \int  \phi  \sum_{k=0}^{u-1} \Big( 1- \frac{k}{u}\Big) \Big( \sum_{\lambda\in \Lambda}   \lambda^k \Phi_\lambda(\gamma) + K^k(1\gamma)\Big)  \ \ dm\\
&=  \int  \phi \Big(\sum_{k=0}^{u-1}K^k(\gamma)\Big) \ dm \\
& +\int \Big( - \frac{1}{u} \sum_{k=0}^{u-1}k K^k(\gamma)  
+  \sum_{\lambda\in \Lambda\setminus\{1\}} \big( \frac{1}{1-\lambda} 
+  \frac{1}{u} \frac{1}{1-\lambda} -\frac{1}{u}\frac{1-\lambda^{u+1}}{(1-\lambda)^2}   \big)  \Phi_\lambda(\gamma)    \Big)  \ dm.
  \end{align*}
  So 
\begin{align*}
&\lim_u T_u(\gamma)=   \int  \phi \Big(\sum_{k=0}^{\infty}K^k(\gamma) +  \sum_{\lambda\in \Lambda\setminus\{1\}}  \frac{1}{1-\lambda}   \Phi_\lambda(\gamma)    \Big)  \ dm\\
&=  \int \gamma  \sum_{n=0}^\infty \sum_{j=0}^{p-1} \phi\circ f^{pn+j} \ dm + G_\gamma,
\end{align*}
where 
$$G_\gamma= \int \phi \Big( \sum_{\lambda\in \Lambda\setminus\{1\}}  \frac{1}{1-\lambda}  \Phi_\lambda(\gamma)    \Big).  $$
Taking $\gamma=1_{[a,x]}$ we conclude the proof. 
\end{proof}

\subsection{Asymptotic variance and regularity of primitives}

\subsubsection{Eigenspaces and spectral projections} Note that $f$ have a finite number of absolutely continuous ergodic probabilities $\mu_\ell=\rho_\ell m$, with $\ell\leq E$ and $\rho_\ell\in BV$,  whose (pairwise disjoint) basins of attractions
$$A_\ell=\{x\in I \ s.t. \ \lim_{N\rightarrow \infty}   \frac{1}{N}  \sum_{k< N}  \theta\circ f^k(x)= \int \theta \ d\mu_\ell, \text{ for every } \theta\in C^0(I)    \}.$$
 covers $m$-almost every point in $I$. Let $$S_\ell= \{x\in I \ s.t. \ \rho_\ell(x) > 0\}\subset A_\ell$$ By Boyarsky and  G\'{o}ra\cite{bg} the set $S_\ell$ is a finite union of intervals up to a zero $m$-measure set. So indeed
 $$A_\ell = \cup_{n=0}^\infty f^{-i}S_\ell$$
 up to  set of zero Lebesgue measure. We have
 $$\Phi_1(\theta)= \sum_{\ell\leq E} \Big( \int_{A_\ell} \theta \ dm\Big)  \rho_\ell.$$
 \\

Let $\psi$ be  in the image of $\Phi_\lambda$, with $\lambda\in \Lambda$. Then $L\psi=\lambda \psi$ and $L|\psi|=|\psi|$. So $|\psi|$ is a non negative linear combination of $\rho_\ell$ , $\ell\leq E$.  Replacing $\psi$ by $\psi 1_{S_\ell}$ we may assume that $|\psi|$ is  a multiple of $\rho_\ell$.

Let $s(x)=\psi(x)/|\psi(x)|$  when $\psi(x)\neq 0$, or zero otherwise. One can see that $\lambda s(f(x))= s(x)$ $m$-almost everywhere. Reciprocally if $s$ is a function such that either $|s(x)|=1$ for $x\in S_\ell$, $s(x)=0$  otherwise, and $\lambda s(f(x))= s(x)$  almost everywhere, then  $L(s\rho_\ell)=\lambda s\rho_\ell$. So define
$$E_{\lambda,\ell}=\{ s\colon I \rightarrow \mathbb{C} \ s.t. \ supp \ s \subset  \ S_\ell \ and \ \lambda s\circ f = s \ on \ S_\ell\}.$$

The ergodicity of $\mu_\ell$ implies that $E_{\lambda,\ell}$ is either zero or one-dimensional. Let $\Lambda^\ell \subset \Lambda$ be such that $\lambda\in \Lambda^\ell$ if and only $\dim E_{\lambda,\ell} =1$. We have that $\Lambda^\ell$ is a finite subgroup of $\mathbb{S}^1$.  

By the previous considerations, if $\lambda\in \Lambda^\ell$  one can choose  an element of  $E_{\lambda,\ell}$, denoted $s_{\lambda,\ell}$, such that  $|s_{\lambda,\ell}|=1$ on $S_\ell$.  Indeed if $\beta$  is a generator of the cyclic group $\Lambda^\ell$, we can  choose $s_{\beta^n,\ell}= (s_{\beta,\ell})^n$ and $\{s_{\beta,\ell}\}_{\beta\in \Lambda^\ell}$ became a cyclic group isomorphic to $\Lambda^\ell$. In particular $s_{\lambda_1,\ell}s_{\lambda_2,\ell}= s_{\lambda_1\lambda_2,\ell}$ and $s_{1,\ell}=1_{S_\ell}$.

If  $s\in E_{\lambda,\ell}$ then  $s\circ f^{p(f)}=s$. So if $\nu$ is an ergodic component  of $\mu_\ell$ for $f^{p(f)}$ we have that $s$ is constant on $supp \ \nu$. By Boyarsky and  G\'{o}ra \cite{bg} the support of $\nu$  is a finite union of intervals up to a zero $m$-measure set. Since the support of the ergodic components  of $\mu_\ell$ cover the support of $\mu_\ell$ we conclude that  $s$ is piecewise constant on $S_\ell$ (and zero elsewhere)  and consequently $E_{\lambda,\ell}\subset BV$. 
So we conclude that 
\begin{equation}\label{cc}  \Phi_\lambda(BV)= \bigoplus_{\ell} \{     s\rho_\ell \colon \  s\in E_{\lambda,\ell}\}= <\{ s_{\lambda,\ell}  \}_{\ell\leq E}>.\end{equation} 
It is convenient to consider a  modification of $s_{\lambda,\ell}$. Define $\hat{s}_{\lambda,\ell}$ as equal to  $s_{\lambda,\ell}$ on $S_\ell$, equals to zero outside $A_\ell$ and  
$$\hat{s}_{\lambda,\ell}(x)=\lambda^n s_{\lambda,\ell}(f^n(x)),$$
where $n$ is some integer satisfying  $f^n(x)\in S_\ell$. It is easy to see that $\hat{s}_{\lambda,\ell}$ is well-defined, $|\hat{s}_{\lambda,\ell}|=1$ and $\lambda \hat{s}_{\lambda,\ell} \circ f = \hat{s}_{\lambda,\ell}$   on $I$. Of course  $\hat{s}_{\lambda,\ell}\in L^\infty(m)$, but it may not belong to $BV$ anymore.

Consider the semi positive definite Hermitian  form 
$$<\gamma_1,\gamma_2>_\ell= \int \gamma_1 \overline{\gamma}_2 \rho_\ell \ dm.$$
for every $\gamma_1,\gamma_2\in BV$. 
Let  $\phi\in BV$ be such that 
$$\int \phi \Phi_1(\gamma) \ dm =0.$$
This is equivalent to
$$\int \phi \rho_\ell \ dm=\int \phi  \hat{s}_{1,\ell} \rho_\ell \ dm=0$$
for every $\ell$. Then one  can find (using Gram–Schmidt process) constants $c_{\lambda,\ell}$   such that  the function  
$$P_\ell(\phi)=\phi 1_{A_\ell}  - \sum_{\lambda\in \Lambda\setminus\{1\}} c_{\lambda,\ell} \hat{s}_{\lambda,\ell}$$
is ortogonal to the subspace
$$\bigoplus_{\lambda\in \Lambda} E_{\lambda,\ell}$$
with respect to this Hermitian form. So we obtained the decomposition
\begin{equation}\label{dec} \phi = \sum_\ell P_\ell(\phi) +  \sum_{\lambda\in \Lambda\setminus\{1\}} c_{\lambda,\ell} \hat{s}_{\lambda,\ell}.\end{equation} 
Let $$P(\phi)=  \sum_{\ell}P_\ell(\phi).$$

\begin{lemma} \label{nulo} Let $\phi, \psi \in BV$ be such that 
$$\int \phi \Phi_1(\gamma) \ dm =0= \int \psi \Phi_1(\gamma) \ dm$$
for every $\gamma\in BV$.
Then for every $\gamma\in BV$ the following holds.
\begin{itemize} 
\item[A.]  For every $\phi, \gamma\in BV$
$$\int P(\phi) \Phi_\lambda(\gamma) \ dm=0.$$
\item[B.]  For $\beta\neq \overline{\lambda}$ we have 
$$\int \hat{s}_{\beta,\ell} \Phi_\lambda(\gamma) \ dm=0.$$
\item[C.]  We have  
$$\int \hat{s}_{\overline{\lambda},\ell} \Phi_\lambda(\gamma) \ dm=\int s_{\overline{\lambda},\ell} \gamma \ dm.$$
\item[D.] For every $\phi, \psi,\gamma\in BV$ and $\lambda\in \Lambda^\ell\setminus \{1\}$ we have
$$\int \phi \Phi_\lambda(\psi \Phi_\lambda(\gamma)) \ dm =0.$$
\end{itemize} 
\end{lemma} 
\begin{proof} By definition 
$$\int P_\ell(\phi)  s\rho_\ell \ dm=0$$
for every $s\in E_{\lambda,\ell}$, $\lambda\in \Lambda\setminus \{1\}$.
Due (\ref{cc}) this is equivalent to 
$$\int P(\phi) \Phi_\lambda(\gamma)  \ dm=0$$ 
for every $\gamma\in BV$. This proves $A$. \\

\noindent Note that the support of  $\Phi_\lambda(\gamma)$ is included in $\cup_\ell S_\ell$, so 

\begin{align*} \int \hat{s}_{\beta,\ell} \Phi_\lambda(\gamma) \ dm&=  \int s_{\beta,\ell} \Phi_\lambda(\gamma) \ dm  =  \lim_N \frac{1}{Np} \int s_{\beta,\ell} \sum_{i=0}^{Np-1} L^i(\gamma) \ dm\\
&=  \lim_N \frac{1}{Np} \int  \gamma \sum_{i=0}^{Np-1}  \frac{s_{\beta,\ell}\circ f^i}{\lambda^i}  \ dm\\
&=  \lim_N \frac{1}{Np} \int  \gamma  s_{\beta,\ell} \sum_{i=0}^{Np-1}   \frac{1}{(\beta\lambda)^i} \ dm.
\end{align*}
Since 
$$\lim_N \frac{1}{Np}  \sum_{i=0}^{Np-1}  \frac{1}{(\beta\lambda)^i}$$
is $1$ if $\beta=\overline{\lambda}$ and $0$ otherwise, we obtained  $B.$ and $C.$ 

To show $D.$ fix $\gamma\in BV$ and $\lambda \in \Lambda^\ell$. The function 
$$(\phi,\psi)\mapsto \int \phi \Phi_\lambda(\psi \Phi_\lambda(\gamma)) \ dm$$
is bilinear. Applying the decomposition (\ref{dec}) to $\phi$  one can see that is enough to show that  the expressions 
\begin{enumerate}
\item $\int P(\phi) \Phi_\lambda(\psi \Phi_\lambda(\gamma)) \ dm$,
\item $\int \hat{s}_{\beta,\ell} \Phi_\lambda(\psi \Phi_\lambda(\gamma)) \ dm$, with $\beta\in \Lambda^\ell$,
\end{enumerate} 
are both zero. By $A.$ we have that $(1)$ vanishes. $B.$ implies that $(2)$ is also zero for $\beta\neq \overline{\lambda}$. Let's consider the case $\beta=\overline{\lambda}$. Then $C.$ implies
$$\int \hat{s}_{\overline{\lambda},\ell} \Phi_\lambda(\psi \Phi_\lambda(\gamma)) \ dm = \int  s_{\overline{\lambda},\ell}\psi \Phi_\lambda(\gamma) \ dm. $$
Since  $\Phi_\lambda(\gamma)$ is a linear combination of elements of $\{s_{\lambda,j}\rho_j\}_{j\leq E}$, it is enough to show that 
$$ \int  s_{\overline{\lambda},\ell}\psi s_{\lambda,j}\rho_j \ dm=0 $$
for every $j$. This is obvious  for  $j\neq \ell$, since in this case  the support of $\rho_j$ is disjoint from the support of $s_{\overline{\lambda},\ell}$. For $j=\ell$ we have that $s_{\overline{\lambda},\ell} s_{\lambda,\ell}=1$ on $S_\ell$, so
$$ \int  s_{\overline{\lambda},\ell}\psi s_{\lambda,\ell}\rho_\ell \ dm = \int  \psi \rho_\ell \ dm =0.$$
\end{proof}

\subsubsection{Asymptotic variance and Livsic cohomological equation} Given a not necessarily invariant probability $\mu$, define 
$$\sigma^2_\mu(\phi)=\lim_{N \rightarrow \infty}  \int \Big| \frac{\sum_{i=0}^{N-1} \phi\circ f^i }{\sqrt{N}}   \Big|^2 \ d\mu.$$
whenever this limit exists. Note that  $\sigma^2_m(\phi)$  is similar to the usual asymptotic variance of $\phi$, but it is not quite the same since $m$ is not necessarily an invariant measure. \\
 
 If  $\phi$ is a piecewise $C^\beta$ function  then it has finite $1/\beta$ bounded variation. In particular it belongs to the space of generalised bounded variations $BV_{1,1/\beta}$ as defined by  Keller \cite[Theorem 3.3]{keller}. In our setting Keller proved that the transfer operator $L$ satisfies the Lasota-Yorke inequality for the pair $(L^1(m), BV_{1,1/\beta})$.  Consequently we can decompose $L$ as in (\ref{ii}), so we keep this same notation for $L$ acting on $BV_{1,1/\beta}$. 
 
 By Keller \cite[Theorem 3.3]{keller} (see also Broise \cite{broise})  we have that  $\sigma^2_{\mu_\ell}(\phi)$ is well-defined and 
\begin{align*} &\sigma^2_{\mu_\ell} (\phi)\\
&=  \lim_{N}  \frac{1}{N}  \int  \Big( 2Re(\sum_{k=0}^{N-1} (N-k)\phi\circ f^k \overline{\phi})   -  N  |\phi|^2   \Big) \ d\mu_\ell, \numberthis \label{vari} \\
&=  2Re\Big(\sum_{\lambda\in \Lambda\setminus \{1\}} \frac{1}{1-\lambda} \int  \phi \Phi_\lambda(\overline{\phi}) \ d\mu_\ell  + \sum_{i=0}^\infty  \int \phi K^i(\overline{\phi}) \ d\mu_\ell \Big)  - \int  |\phi|^2  \  d\mu_\ell.   \end{align*} 

The following result is a generalization of well-known results on asymptotic variance for invariant measures (see Broise \cite{broise}). The main difference is that we obtain regularity results in the whole phase space and not just on the support of an invariant measure. 

\begin{theorem}\label{cohomo}  Let $f\in \mathcal{B}^{1+BV}_{exp}(C)$ expanding map on the interval $I=[a,b]$. 
Let $p$ be  a multiplier of $p(f).$   Let $\phi$ and $\psi$   be functions  in $\mathcal{B}^\beta(C)$, with $\beta \in (0,1)$,   such that 
$$\int \phi \Phi_1(\gamma) \ dm =\int \psi \Phi_1(\gamma) \ dm=0$$
for every $\gamma\in BV$.  Then the limit 
$$\sigma_m(\phi,\psi)=  \lim_{N \rightarrow \infty}  \int \Big( \frac{\sum_{i=0}^{N-1} \phi\circ f^i }{\sqrt{N}}   \Big)\Big( \frac{\sum_{i=0}^{N-1} \overline{\psi} \circ f^i }{\sqrt{N}}   \Big)   \ dm$$
exists. In particular $\sigma^2_m(\phi)= \sigma_m(\phi,\phi)$ is well-defined and   $\sigma_m$ is a positive semidefinite Hermitian form.  We also have
\begin{equation}\label{variance}  \sigma^2_m(\phi) =\sum_\ell m(A_\ell) \sigma^2_{\mu_\ell} (\phi).\end{equation} 
Furthermore the   following statements  are equivalent 
\begin{itemize}
\item[A.] $\sigma^2_m(\phi)=0$.
\item[B.] We have
$$\sup_{N} |\sum_{i=0}^{N-1} \phi\circ f^i   |_{L^2(m)}< \infty.$$
\item[C.]  There is $g\in L^2(m)$ such that $g$ is the weak limit in $L^2(m)$  of the sequence
$$T_M(\phi)=-\frac{1}{M} \sum_{N=0}^{M-1} \sum_{k=0}^{N-1}  \phi\circ f^{k}.$$
In particular the function
\begin{equation} \label{rew} \alpha(x) = \lim_{M\rightarrow \infty}   - \frac{1}{M} \sum_{N=0}^{M-1}   \int 1_{[a,x]} \Big( \sum_{k=0}^{N-1}  \phi\circ f^{k}\Big) \ dm\end{equation} 
is absolutely continuous, $1/2$-H\"older continuous,  and its derivative is  $g$.
\item[D.]  There is $g  \in L^2(m)$  that satisfies
\begin{equation*}\phi = g \circ f -g\end{equation*}
$m$-almost everywhere in $I$.
\item[E.] For every $\ell\leq E$ there is $g_\ell\in L^\infty(m)$ such that 
\begin{equation*}\phi = g_\ell \circ f -g_\ell\end{equation*}
on $S_\ell$.
\end{itemize}
Moreover $A-E$ implies 
\begin{itemize}
\item[F.] For every periodic point $q\in  \hat{I}\cap \overline{S}_\ell$, with $\ell\leq E$ and $f^m(q)=q$ we have 
$$\sum_{j=0}^{m-1} \phi(f^j(q))=0.$$
Note that we need to consider lateral limits if $\phi$ is not continuous at some points in the orbit of $q$.
\end{itemize} 
\end{theorem} 

\begin{rmk} We believe that $F$ is indeed equivalent to $A$-$E$, but we did not manage to prove it. 
\end{rmk} 

\begin{proof} To study  $\sigma^2_m(\phi)$ we will use methods similar to the study of the usual  asymptotic variance (see  Broise \cite{broise}), however the non invariance of $m$  turns things a little more cumbersome.  Note that
$$\int \phi \ d\mu_\ell=0=\int \psi \ d\mu_\ell$$
for every $\ell$.  Denote
$$\sigma_{m,N}(\phi,\psi)=   \int \Big( \frac{\sum_{i=0}^{N-1} \phi\circ f^i }{\sqrt{N}}   \Big)\Big( \frac{\sum_{i=0}^{N-1} \overline{\psi} \circ f^i }{\sqrt{N}}   \Big)   \ dm.$$
Of course $\sigma_N$
is linear in $\phi$ and antilinear in $\psi$ and consequently   it satisfies the polarization identity 
$$\sigma_{m,N}(\phi,\psi)= \frac{1}{4} ( \sigma^2_{m,N}(\phi+\psi)-  \sigma^2_{m,N}(\phi-\psi)+ i  \sigma^2_{m,N}(\phi+i \psi)- i \sigma^2_{m,N}(\phi-i\psi),$$
so it is enough to show that $\sigma_m^2(\phi)= \lim_N \sigma_{m,N}^2(\phi)$ exists.  Note that 
\begin{align*} &\int \Big| \frac{\sum_{i=0}^{N-1} \phi\circ f^i}{\sqrt{N}}   \Big|^2   \ dm \\
  &= \frac{1}{N}   \int  \sum_{i< N} \sum_{j < N}  \phi\circ f^i  \  \overline{\phi}\circ f^j\ dm \\
   &=  \frac{1}{N}  \int  \Big( 2 Re\big(\sum_{i\leq j< N}  \phi\circ f^i  \  \overline{\phi}\circ f^j\big)   -  \sum_{i< N} |\phi|^2 \circ f^i  \Big) \ dm \\
   &= \frac{1}{N}   \int  \Big( 2 Re\big(\sum_{i\leq j< N}  \phi L^{j-i}( \overline{\phi} L^i1_I)\big)   -  \sum_{i< N}  |\phi|^2 L^i 1_I \Big) \ dm\\
      &= \frac{1}{N}   \int  \Big( 2 Re\big(\sum_{i< N} \sum_{k=0}^{N-1-i} \phi L^{k}( \overline{\phi} L^i1_I)\big)   -  \sum_{i< N}  |\phi|^2 L^i 1_I \Big) \ dm\\
            &= \frac{1}{N}   \int  \Big( 2 Re\big(\sum_{i< N} \sum_{k=0}^{N-1-i} \phi L^{k}( \overline{\phi} \Phi_1(1_I))\big)   -  \sum_{i< N}  |\phi|^2 L^i  \Phi_1(1_I)  \Big) \ dm+R_N\\
            &= \frac{1}{N}   \int  \sum_{\ell=1}^E m(A_\ell) \Big( 2 Re\big(\sum_{i< N} \sum_{k=0}^{N-1-i} \phi L^{k}( \overline{\phi} \rho_\ell)\big)   -  \sum_{i< N}  |\phi|^2  \rho_\ell \Big) \ dm+ R_N\\
                 &= \frac{1}{N}   \sum_{\ell=1}^E   m(A_\ell) \int  \Big( 2 Re\big( \sum_{k=0}^{N-1} (N-k)\phi\circ f^k \overline{\phi}\big)   -  N  |\phi|^2  \Big) \ d\mu_\ell+ R_N,
              \numberthis \label{mn} 
  \end{align*} 
where  
\begin{align*} N R_N&=    \int  \Big( 2 Re\Big( \sum_{i< N} \sum_{k=0}^{N-1-i} \phi  \Big( \sum_{\hat{\lambda} \in \Lambda\setminus\{1\}} \hat{\lambda}^k \Phi_{\hat{\lambda}} \Big( \overline{\phi} \cdot \big( \sum_{\lambda \in \Lambda\setminus \{1\}} \lambda^i \Phi_\lambda(1_I) \big)  \Big) \Big) \Big) \ dm \\
&+ \int  \Big( 2 Re\big( \sum_{i< N} \sum_{k=0}^{N-1-i} \phi K^{k}\Big( \overline{\phi} \cdot \big( \sum_{\lambda \in \Lambda\setminus \{1\}} \lambda^i \Phi_\lambda(1_I) \big)  \Big)\Big)  \ dm \\
&+  \int  \Big( 2Re\big(\sum_{i< N} \sum_{k=0}^{N-1-i} \phi \Big( \sum_{\lambda\in \Lambda} \lambda^k \Phi_\lambda \Big( \overline{\phi} \cdot \big(  K^i(1_I) \big)  \Big) \Big) \Big) dm\\
&+  \int  \Big( 2 Re\Big( \sum_{i< N} \sum_{k=0}^{N-1-i} \phi K^{k}\Big( \overline{\phi} \cdot \big(  K^i(1_I) \big) \Big) \Big)  dm\\
&  -  \sum_{i< N} |\phi|^2 \cdot \Big(  \sum_{\lambda \in \Lambda\setminus \{1\}} \lambda^i \Phi_\lambda(1_I)  \Big)   \ dm \\
& -  \sum_{i< N} |\phi|^2 \cdot K^i(1_I)   \ dm.  \end{align*} 
We have
\begin{align*}   & \int  \Big( \sum_{i< N} \sum_{k=0}^{N-1-i} \phi  \Big( \sum_{\hat{\lambda} \in \Lambda\setminus \{1\}} \hat{\lambda}^k \Phi_{\hat{\lambda}} \Big(\overline{\phi}  \cdot \big( \sum_{\lambda \in \Lambda\setminus \{1\}} \lambda^i \Phi_\lambda(1_I) \big)  \Big) \Big) \ dm \\
&=    \int  \Big( \sum_{i< N}  \sum_{\hat{\lambda} \in \Lambda\setminus \{1\}}  \Big(  \sum_{k=0}^{N-1-i}  \hat{\lambda}^k \Big) \phi  \cdot \Big(  \Phi_{\hat{\lambda}} \Big(\overline{\phi} \cdot \big( \sum_{\lambda \in \Lambda\setminus \{1\}} \lambda^i \Phi_\lambda(1_I) \big)  \Big) \Big) \ dm  \\
&=    \int  \Big( \sum_{i< N}  \sum_{\hat{\lambda} \in \Lambda\setminus \{1\}}  \frac{1-\hat{\lambda}^{N-i}}{1-\hat{\lambda} }  \phi  \cdot   \Big(  \Phi_{\hat{\lambda}} \Big(\overline{\phi} \cdot \big( \sum_{\lambda \in \Lambda\setminus \{1\}} \lambda^i \Phi_\lambda(1_I) \big)  \Big) \Big) \ dm  \\
&=  \sum_{\hat{\lambda} \in \Lambda\setminus \{1\}}  \sum_{\lambda \in \Lambda\setminus \{1\}}   \frac{1}{1-\hat{\lambda} }  \frac{1-\lambda^N}{1-\lambda } \int  \phi  \cdot   \Big(  \Phi_{\hat{\lambda}} \Big(\overline{\phi}  \cdot \big( \Phi_\lambda(1_I) \big)  \Big)\ dm  \\
&+   \sum_{\hat{\lambda} \in \Lambda\setminus \{1\}} \sum_{\hat{\lambda} \in \Lambda\setminus \{1\}} - \frac{\hat{\lambda}^{N}}{1-\hat{\lambda} } \Big( \sum_{i< N}  \big( \frac{\lambda}{\hat{\lambda}}\big)^i \Big)   \int  \phi   \Big(  \Phi_{\hat{\lambda}} \Big( \overline{\phi} \cdot \big(  \Phi_\lambda(1_I) \big)  \Big)  \ dm  \\
&=   -N\sum_{\lambda \in \Lambda\setminus \{1\}}    \frac{\lambda^{N}}{1-\lambda }   \int  \phi   \Big(  \Phi_{\lambda} \Big(\overline{\phi}  \cdot \big(  \Phi_\lambda(1_I) \big)  \Big)  \ dm+ O(1)\\
&=O(1).\end{align*} 

The last passage follows from Lemma \ref{nulo}.D. A careful analysis of the (simpler) remaining  terms of $NR_N$  gives us
\begin{equation}\label{vbvb} NR_N= O(1). \end{equation} 
Moreover note that 
\begin{equation}\label{eee} \lim_N \frac{1}{N}   \sum_{\ell=1}^E   m(A_\ell) \int  \Big( 2Re\Big(\sum_{k=0}^{N-1} (N-k)\phi\circ f^k \overline{\phi} \Big)  -  N |\phi|^2 \Big) \ d\mu_\ell =\sum_\ell m(A_\ell) \sigma^2_{\mu_\ell} (\phi).\end{equation} 
So (\ref{mn}) and (\ref{vbvb})  imply
\begin{align*}& \sigma^2_m(\phi)=\lim_{M\rightarrow \infty} \frac{1}{M} \sum_{N< M}   \int \Big( \frac{\sum_{i=0}^{N-1} \phi\circ f^i }{\sqrt{N}}   \Big)^2 \ dm\\
&= \sum_\ell m(A_\ell) \sigma^2_{\mu_\ell} (\phi)+ \lim_M \frac{1}{M} \sum_{N< M} \sum_{\hat{\lambda} \in \Lambda\setminus \{1\}}    \frac{\hat{\lambda}^{N}}{1-\hat{\lambda} }   \int  \phi  \cdot  \Big(  \Phi_{\lambda} \Big( \phi \cdot \big(  \Phi_\lambda(1_I) \big)  \Big)  \ dm\\
&= \sum_\ell m(A_\ell) \sigma^2_{\mu_\ell} (\phi)+ \lim_M \frac{1}{M} \sum_{\hat{\lambda} \in \Lambda\setminus \{1\}}    \frac{1-\hat{\lambda}^{M}}{1-\hat{\lambda} }   \int  \phi  \cdot  \Big(  \Phi_{\lambda} \Big( \phi \cdot \big(  \Phi_\lambda(1_I) \big)  \Big)  \ dm\\
&= \sum_\ell m(A_\ell) \sigma^2_{\mu_\ell} (\phi).
\end{align*} 
This proves that $\sigma^2(\phi)$ is well defined.  \\

\noindent $A\implies B.$ If $A.$ holds, then $\sigma^2_{\mu_\ell}(\phi)=0$ for every $\ell$ and we can  use  the same method as in Broise \cite[Lemma 6.2]{broise} to prove that $B.$ holds. \\

\noindent $B\implies C.$ We use the same methods as Broise \cite{broise}. Theorem  \ref{conv2} imply \begin{align*} & \lim_{M} \frac{1}{M} \sum_{n=0}^{M-1}  \int  \gamma \cdot \Big( \sum_{k=0}^{N}  \phi\circ f^{k} \Big)\ dm\\
& =  \int \gamma \sum_{N=0}^\infty \sum_{j=0}^{p-1} \phi\circ f^{pn+j} \ dm 
\end{align*} 
for every $\gamma\in BV$. But $B.$ implies 
$$\sup_M  | \frac{1}{M} \sum_{n=0}^{M-1} \sum_{k=0}^{N}  \phi\circ f^{k}|_{L^2(m)}< \infty.$$
Since $BV$ is dense in $L^2(m)$ we conclude that there is $g \in L^2(m)$ such that 
$$w-lim_M -\frac{1}{M} \sum_{n=0}^{M-1} \sum_{k=0}^{N-1}  \phi\circ f^{k}  = g,$$
where $w-lim$ denotes the limit in the weak topology of $L^2(m)$. 
 Note that $T_M(\phi\circ f)= T_M(\phi)\circ f$.  For every $w\in BV$ we have
\begin{align*}  &\int ( L \gamma-\gamma) \cdot T_M(\phi)  \ dm \\
&= \int \gamma\cdot (T_M(\phi)\circ f - T_M(\phi)) \ dm \\
&= \int  \gamma\cdot  (\frac{1}{M} \sum_{N< M}  (\phi- \phi\circ f^{N})) \ dm \\
&=\int \gamma \phi  dm - \frac{1}{M}   \int  \gamma\cdot  ( \sum_{N< M}  \phi\circ f^{N} \ dm) \ dm
\end{align*} 
Taking the limit on $M$ we obtain
\begin{align*}  &\int  \gamma (g\circ f - g) \ dm = \int ( L \gamma-\gamma) \cdot g  \ dm =\int \gamma \phi  \ dm.
\end{align*} 
For every $\gamma\in BV$. It easily follows that $g\circ f- g=\phi$.\\ \\
\noindent $D\implies E.$  Consider the spaces $BV_{p,\beta}$ as in Keller \cite{keller}.  Since $\rho_\ell\in BV$ and $\inf_{S_\ell} \rho_\ell > 0$ (see Boyarsky and  G\'{o}ra \cite[Proposition 8.2.3]{bg} ) we have  $1_{S_\ell}/\rho_\ell \in BV\subset  BV_{1,1}\subset BV_{1,\beta}$.  Since $BV_{1,\beta}$ is a Banach algebra (Saussol \cite[Proposition 3.4]{Saussol}) we can consider the normalised transfer operator
$$P\colon BV_{1,\beta} \rightarrow BV_{1,\beta}$$ given by
$$P(w)=   \frac{1_{S_\ell}}{\rho_\ell}  L(w\rho_\ell).$$

Using the same argument as in Broise \cite[Lemme 6.6]{broise} with the transfer operator acting on  $BV_{1,\beta}$ instead of acting on $BV$ one can prove that $g \rho_\ell \in BV_{1,\beta}$. Since   $1_{S_\ell}/\rho_\ell \in  BV_{1,\beta}$ and  $BV_{1,\beta}$ is a Banach algebra  we get   $1_{S_\ell} g \in BV_{1,\beta}\subset L^\infty(m)$. This completes the proof. \\

\noindent $E\implies A.$ It is enough to show that $\sigma^2_{\mu_\ell}(\phi)=0$ for every $\ell$. This  is an easy and well-known argument. \\

\noindent $A-E\implies F$. Fix $\ell_0\leq E$ and consider $g_{\ell_0}$ as in $E$.  Define the function $\hat{g}\colon I \rightarrow \mathbb{R}$ as $1$ outside $\overline{S}_{\ell_0}$ and equals to $g_{\ell_0}$ inside $\overline{S}_{\ell_0}$. Recall that $\overline{S}_{\ell_0}$ is a finite union of intervals. Then
$$h(x)=\int 1_{[a,x]} e^{\hat{g}} \ dm$$
is a Lipschitz function with a Lipschitz inverse.  Define
$$\hat{f}\colon h(\overline{S}_{\ell_0})\rightarrow h(\overline{S}_{\ell_0})$$
by
$$\hat{f}=    h\circ f\circ h^{-1}$$
At first glance  one can see that $\hat{f}$ is piecewise Lipschitz. We claim that $\hat{f}$ is indeed piecewise $C^{1+\beta}$. Note that 
\begin{align*} D\hat{f}(x)&= Dh\circ f\circ h^{-1}(x) \cdot Df\circ h^{-1}(x)\cdot Dh^{-1}(x)\\
&=e^{g_{\ell_0}\circ f\circ h^{-1}(x) } Df\circ h^{-1}(x) e^{-g_{\ell_0}\circ h^{-1}(x)}\\ 
\label{deri} &= e^{\phi\circ  h^{-1}(x)}Df\circ h^{-1}(x),\numberthis \end{align*}
so $D\hat{f}$ is piecewise $C^\beta$ and its discontinuities  belong to $h(C)$.  Let $q\in \overline{S}_{\ell_0}$ be a periodic point, $f^m(q)=q$. Choose $\delta > 0$ such that $Df^m$ do not have discontinuities on $I_0=[q,q+\delta]$ (we can do the same argument for $I_0=[q,q-\delta]$). Then $f^m\colon I_0 \mapsto f^m(I_0)$ has an $C^{1+\beta}$ inverse, denoted by $T$. Let $I_j=T^j(I_0)$. By the mean value theorem and the expansion of $f$ there is $\Cll{345f}> 0$ such that for every $j$
$$  \frac{1}{\Crr{345f}} |Df^{mj}(q)|  \leq    \frac{|I_0|}{|I_j|} \leq \Crr{345f} |Df^{mj}(q)|,   $$
so $$ |Df^{mj}(q)|= \lim_j |I_j|^{-1/j}.$$
Noe that $\hat{f}^m(h(q))=h(q)$.  Since $\hat{f}$ is  piecewise $C^{1+\beta}$ we can do the very same analysis considering $\hat{I}_j=h(I_j)$ and conclude that
$$ |D\hat{f}^{mj}(h(q))|= \lim_j |h(I_j)|^{-1/j}.$$
Since $h$ and its inverse are Lipschtiz there is $\Cll{lippp}> 1$ such that  for every $j$
$$  \frac{1}{\Crr{lippp}}\leq   \frac{h(I_j)}{I_j}\leq \Crr{lippp}.$$
So $\lim_j |h(I_j)|^{-1/j}=\lim_j |I_j|^{-1/j}$ and consequently $D\hat{f}^{mj}(h(q))=Df^{mj}(q)$. 
By (\ref{deri}) this implies 
$$Df^{mj}(q)=Df^{mj}(q)\Pi_{j=0}^m e^{\phi(f^j(q))},   $$
and $F$ follows. 
\end{proof}

Let $\mathcal{P}^n$ be the partition of $I$ by   the open  intervals of monotonicity of $f^n$, that is, $J=(a,b) \in \mathcal{P}^n$ if $f^i(J)\cap C=\emptyset$ for every $i< n$ and there is $i_a,i_b< n$ such that $\{f^{i_a}(a^+),f^{i_b}(b^-)\} \subset \hat{C}$.

\begin{lemma} \label{bvpartition} There is $\Cll{dist} >0$ and $\Cll{distp} > 0$  such that  for every  $n\in \mathbb{N}$ and $J\in \mathcal{P}^n$  we have 
\begin{equation}\label{ddis}   \frac{1}{ \Crr{dist}} \leq    \frac{Df^n(x)}{Df^n(y)}\leq \Crr{dist}\end{equation} 

\begin{equation}\label{ddis2}  |\ln  |Df^n(x)|  - \ln |Df^n(y)| |\leq \Crr{distp}|f^n(x)-f^n(y)|.\end{equation} 
for all $x,y\in J$.  Moreover if $\gamma$ is a bounded variation function with support contained in $J$ we have that for every $x\in J$ 
\begin{equation}\label{bvest}  v(L_F^n(\gamma))\leq \frac{\Crr{dist} }{|DF^n(x)|} \Big( v(\gamma)+  \Crr{distp} |I|   |\gamma|_{L^\infty(m)}\Big),\end{equation} 
where $v(g)$ denotes the variation of $g$.

\end{lemma}

\subsection{Modulus of continuity: Statistical properties} 
Indeed when $\sigma_m(\phi) > 0$ the behavior of primitives of Birkhoff sums can be very wild on at least some part of the phase space, as proved in  de Lima and S. \cite{central} for expanding maps of the circle. We extend some of those results to the setting of piecewise expanding maps.

\begin{theorem}\label{cite} Let $f\in \mathcal{B}_{exp}^{2+\beta}(C)$ and $\phi\in  \mathcal{B}^{\beta}(C)$, with $\beta\in (0,1)$, such that 
$$\int \phi \Phi_1(\gamma) \ dm=0$$
for every $\gamma\in BV$.  Suppose $\sigma_{\mu_\ell}(\phi) > 0$ for some $\ell\leq E$. Let 
 \begin{equation*} \psi(x)=   \int  1_{[a,x]}\cdot \Big( \sum_{k=0}^\infty \sum_{j=0}^{p-1}\phi\circ f^{kp+j} \Big)\ dm.\end{equation*}
Then we have
 $$\lim_{h\rightarrow  0}  \mu_\ell\{ x\in I \colon  \frac{1}{  \sigma_{\mu_\ell}(\phi) L_\ell \sqrt{-\log |h|}}   \Big( \frac{\psi(x+h)-\psi(x)}{h}\Big) \leq y   \} =\frac{1}{2\pi} \int_{-\infty}^y e^{-x^2} \ dx.   $$
Here
$$L_\ell = \Big(\int |Df| \ d\mu_\ell\Big)^{-1/2}.$$
In particular $\psi$ is not a Lipschitz function of any measurable subset of positive measure in the support of $\mu_\ell$ and $\psi$ does not have bounded variation  on the support of $\mu_\ell$.
\end{theorem} 

For $m$-almost every point $x\in I$ and $h> 0$ small we define $N(x,h)$ as  the integer such that
$$ \frac{1}{|Df^{k+1}(x)|} \leq     |h|\leq \frac{1}{|Df^k(x)|} .$$

We are going to need

\begin{proposition}\label{prolima}  For every $\gamma> 0$ there is  $\delta > 0$ with the following property.  For every small $h_0> 0$, $h'\leq h_0$  one can find sets  $\Gamma^\delta_{h',h_0}\subset \Gamma^\delta_{h_0}$ such that 
\begin{itemize}
\item[A.]  We have $m(\Gamma^\delta_{h_0})\geq 1-\gamma$. 
\item[B.]  If $h'\leq \hat{h}$ then $\Gamma^\delta_{\hat{h},h_0} \subset \Gamma^\delta_{h',h_0}.$
\item[C.]   $\lim_{h'\rightarrow 0}  m(\Gamma^\delta_{h',h_0})=  m(\Gamma^\delta_{h_0})$. 
\item[D.]   There is $\Crr{upper}> 1$ and $\mathcal{K} > 0$ such that for  every  $x\in \Gamma^\delta_{h,h_0}$ and $h< h'$   there is  $N_1(x,h)$ satisfying 
 $$  N(x,h)- \mathcal{K}\log N(x,h) \leq   N_1(x,h) \leq N(x,h)$$
 such that if $\omega_{x,h}$ is defined by  $x\in \omega_{x,h}\in \mathcal{P}^{k}$  then 
 $$ \frac{1}{\Crr{upper}} \frac{1}{|Df^{N_1(x,h)}(y)|} \leq  |\omega_{x,h}|\leq \Cll{upper} \frac{1}{|Df^{N_1(x,h)}(y)|},$$
$$|f^{N_1(x,h)}(\omega_{x,h})|\geq  \delta,$$
$$|Df^{N_1(x,h)}(x)| |\omega_{x,h}|\geq  \delta,$$
and moreover $[x,x+h]$ is in the interior of $\omega_{x,h}$. 
\end{itemize} 
\end{proposition} 
\begin{proof} The proof of this result it is quite similar  to  a related result in de Lima and S. \cite[Proposition 4.5]{lima2}. Indeed it is easier since we all dealing with the phase space instead of the parameter space as  in that reference.  
\end{proof}

\begin{proof}[Proof of \ref{cite}] The proof is quite similar  to the proof of the main results in de Lima and S. \cite{lima2}, so we detail only the main distinctions. Our setting here is actually easier  (we deal with the phase space instead of the parameter space).  

Let $x\in \Gamma^\delta_{h',h_0}$ and $h < h'$.  Let $N_1(x,h)$ as in Proposition \ref{prolima} and write  $N_1(x,h)= ap(f)+ r$, where $0\leq r < p=p(f)$. 
Then $f^i[x,x+h]\cap C =\emptyset$ and $\phi$ is $\beta$-H\"older on $f^i[x,x+h]$  for every $i<  ap(f)$.  This implies
\begin{align*} &\int 1_{[x,x+h]}\sum_{M=0}^{a-1} \sum_{j=0}^{p-1} \phi\circ f^{Mp+j}   \ dm 
                   =  \Big( \sum_{M=0}^{a-1} \sum_{j=0}^{p-1} \phi\circ f^{Mp+j}(x) \Big)  h  + R(x,h),  \end{align*} 
where
\begin{align*} 
|R_{x,n}|&=\Big| \int 1_{[x,x+h]}(y)\sum_{M=0}^{a-1} \sum_{j=0}^{p-1} \phi\circ f^{Mp+j}(y)-\phi\circ f^{Mp+j}(x)   \ dm(y) \Big|\\
&\leq \int 1_{[x,x+h]}(y)\sum_{M=0}^{a-1} \sum_{j=0}^{p-1} |\phi\circ f^{Mp+j}(y)-\phi\circ f^{Mp+j}(x)|   \ dm(y)  \\
&\leq \Cll{nnn}  \Big(  \int 1_{[x,x+h]}(y)  \sum_{M=0}^{a-1} \sum_{j=0}^{p-1}   | f^{Mp+j}(y) -  f^{Mp+j}(x)|^\beta \ dm(y) \Big) \\
\label{aux0} &\leq  \Crr{nnn} |h| (\sum_{j=0}^{ap-1} (\inf |Df|)^{- \beta j} |I|^\beta \leq \Cll{nnnn} |h|. \numberthis
\end{align*} 
Note that due Lemma \ref{bvpartition} we have that there is a $\Cll{oyoy}$ such that 
$$bv(L^{ap}1_{[x,x+h]})\leq \frac{\Crr{oyoy}}{|Df^{ap}(x)|} ,$$
$$|L^{ap}1_{[x,x+h]}|_{L^1(m)}=|h|,$$
Let $N_2(x,h)$ be  the smallest  integer  divisible by $p$ such that 
$$\frac{\Crr{c}^{N_2(x,h)} }{|Df^{ap}(x)|}\leq |h|.$$
Then Lasota-Yorke inequality gives us
$$|L^{ap+N_2(x,h)}1_{[x,x+h]}|_{BV}\leq \Cll{12345} |h|,$$
so we can use the same argument as in the proof of Theorem \ref{vvv} to conclude that
 \begin{align*} \Big| \int 1_{[x,x+h]}  \sum_{j=0}^\infty \phi\circ f^{ap+N_2(x,h)} \Big| &= \Big|  \int \phi \sum_{j=0}^\infty L^{ap+N_2(x,h)}1_{[x,x+h]} \ dm \Big|\\
 \label{aux1}  &\leq \Cll{12345er} |h|.\numberthis \end{align*} 
Note that 
\begin{align*} |h|&\geq \frac{1}{|Df^{N(x,h)+1}(x)|}= \frac{1}{|Df^{N(x,h)+1-ap}(f^{ap}(x) )|} \frac{1}{|Df^{ap}(x)|}\\
&\geq \Cll{tyty} \frac{\Cll[c]{q12w}^{\mathcal{K} \log N(x,h)}}{|Df^{ap}(x)|} \geq \Crr{tyty} \frac{\Crr{c}^{\frac{\log \Crr{q12w}}{\log \Crr{c}}\mathcal{K} \log N(x,h)}}{|Df^{ap}(x)|}.  \\
 \end{align*} 
Here $\Crr{q12w}= (\sup |Df|)^{-1}$.  So  
 $$N_2(x,h)\leq \Cll{efef} \log N(x,h) +  \Cll{efef2}. $$
 and
 \begin{equation}\label{aux2} \Big| \int  \phi \sum_{j=ap}^{N_2(x,h)-1} L^j1_{[x,x+h]} \ dm \Big| \leq  (\Crr{efef} \log N(x,h) +  \Crr{efef2})|h|. \end{equation} 
Finally note that 
\begin{equation}\label{aux3}  \log N(x,h)\leq \Cll{sunew} \log \log (\frac{1}{|h|}).\end{equation} 
Putting together the estimates (\ref{aux0}), (\ref{aux1}), (\ref{aux2}) and (\ref{aux3}) we obtain
\begin{equation}\label{cltlim}  \frac{\psi(x+h)-\psi(x)}{h}= \sum_{j=0}^{N_1(x,h)}\phi\circ f^j(x) + O(\log \log (\frac{1}{|h|}))\end{equation} 
for every  $x\in \Gamma^\delta_{h',h_0}$ and $h < h'$. By Keller \cite[Theorem 3.3]{keller} we have that 
$\phi$ satisfies the Functional Central Limit Theorem, that is, if we define $Y_n(\theta,x)$, with $x\in I$, by 
$$Y_N(\theta,x)= \frac{1}{\sigma_{\mu_\ell}\sqrt{N}}\sum_{j=0}^{\floor{N\theta}-1} \phi(f^j(x))+  \frac{N\theta- \floor{N\theta}}{\sigma_{\mu_\ell}\sqrt{N}} \phi(f^{\floor{N\theta}}(x)),  $$
then $Y_N$ converges in distribution (considering the measure $\mu_\ell$)  to the Wiener measure.  Moreover one can easily verify that 
$$-\frac{N_1(x,h)}{\log |h|}$$
converges in distribution (considering the measure $\mu_\ell$) to 
 $L^{-1}=(\int \ln |Df| \ d\mu_\ell)^{-1}$. Using (\ref{cltlim}) and classical tools in Probability one can complete the proof of the central Limit Theorem for the modulus of continuity of $\psi$.  See de Lima and S. \cite[Section 5]{lima2} for details. The fact that $\psi$ is not a Lipschitz function on any subset of positive measure in the support of $\mu_\ell$ also   follows in the same way as a similar result there. See de Lima and S. \cite[Section 9]{lima2}. \\
 
 \noindent   It is a simple exercise to show that if $\psi$ has bounded variation, then for every $\epsilon > 0$ there is a set $\Omega_\epsilon$ with $m(\Omega_\epsilon) > m(S_\ell)- \epsilon$  such that $\psi$ is Lipschitz in $\Omega_\epsilon$ (we have even a Lusin type result for bounded variation functions. See Goffman and Fon Che \cite{goff}). So $\psi$ does not have bounded variation.
\end{proof}  

\begin{rmk} In an abstract setting of a compact metric space $X$  with a reference measure $m$ and a non-singular dynamics $f\colon X\rightarrow X$,   one  can ask if similar results holds. That is, if the "distribution" $\psi$ defined by a Birkhoff sum 
$$\psi(A)= \lim_ N \sum_{i=0}^{N-1} \int   \phi\circ f^i(x) A(x)\ dm(x)$$
is well-defined  when the function $A$ is the  characteristic function of a ball, so  can ask about the distributional limit of the random variables
$$\frac{\psi(1_{B(x,r)}) }{m(1_{B(x,r)})}$$
considering the reference measure $m$, when $r$ goes to zero, after proper normalisation. Such study for the full shift and its Gibbs measures, for instance, would be interesting. This topic is somehow related (but not quite the same)   to Leplaideur and Saussol \cite{ls}. \end{rmk}

\subsection{Bounded variation regularity is rare} One can ask if Birkhoff sums can be more regular than Log-Lipschitz. In this section, we are going to see that  bounded variation  regularity is very rare.

\begin{theorem}\label{lipsc} Let $f\in \mathcal{B}^{2+\beta}(C)$, with $\beta\in (0,1)$. 
Let $p$ be  a multiplier of $p(f).$   Let $\phi\in \mathcal{B}^\beta(C)$ be  such that 
\begin{equation}\label{ficod} \int \phi \Phi_1(\gamma)=0\end{equation} 
for every $\gamma \in BV$. Consider
 \begin{equation*} \psi(x)=   \int  1_{[a,x]}\cdot \Big( \sum_{k=0}^\infty \sum_{j=0}^{p-1}\phi\circ f^{kp+j} \Big)\ dm.\end{equation*}
Then the following statements are equivalent
\begin{itemize}
\item[A.]  $\psi$ has bounded variation on each $S_\ell$, with $\ell \leq E$. 
\item[B.]  There is a  function  $g\in L^2(I)$ such that 
\begin{equation}\label{tyu} \phi = g\circ f -g\end{equation}
and $g\in L^\infty(S_\ell)$ for every $\ell\leq E$.  
\item[C.]  $\psi$ is absolutely continuous, $1/2$-H\"older continuous on $I$, and it  is Lipschitz on each $S_\ell$, with $\ell \leq E$. 
\end{itemize} 
Moreover if $A-C$ hold  then 
\begin{itemize}
\item[D.] for  every periodic point $q\in  \hat{I}\cap \overline{S}_\ell$, with $\ell\leq E$ and $f^m(p)=p$ we have 
\begin{equation}\label{inficod} \sum_{j=0}^{m-1} \phi(f^j(q))=0.\end{equation} 
Note that we need to consider lateral limits here if $\phi$ is not continuous at some points in the orbit of $q$.
\end{itemize}
\end{theorem} 

\begin{rmk} The condition (\ref{ficod}) is satisfied in  a subspace of $\mathcal{B}^\beta(C)$ with {\it  finite} codimension, but condition (\ref{inficod}) on all periodic points in $\cup_\ell \overline{S}_\ell$  is satisfied only in a subspace with {\it  infinite}  codimension. This justifies  the claim that bounded variation regularity of $\psi$ is "rare". 
\end{rmk} 

\begin{proof} Of course $C\implies A$. \\

\noindent {\it $A\implies B \ and \ C$.} Suppose that $\psi$ has bounded variation. We claim that $\sigma^2_m(\phi)=0$. Indeed, suppose that this is not true. Then  $\sigma^2_{\mu_{\ell_0}}(\phi)>0$ for some $\ell_0$. By Theorem \ref{cite} we have that $\psi$ does not have bounded variation on the support of $\mu_{\ell_0}$.   So we conclude that $\sigma^2_m(\phi)$ must be zero. Theorem \ref{cohomo} says that $\alpha$, as defined in (\ref{rew}), is absolutely continuous on $I$, and its derivative $D\alpha\in L^2(m)$ satisfies $\phi= D\alpha \circ f - D\alpha$ and $\D\alpha \in L^\infty(S_\ell)$ for every $\ell\leq E$. Moreover $\alpha$ is absolutely continuous, $1/2$-H\"older continuous on $I$, and it  is Lipschitz on each $S_\ell$, with $\ell \leq E$.  By Theorem \ref{conv2} we have that $\alpha = \psi+ G$ is a  Lipschitz  function, so $D\alpha=D\psi+DG\in L^2(I)$ and $D\alpha\in L^\infty(S_\ell)$ for every  $\ell\leq E$. This completes the proof. \\

\noindent {\it $B\implies C$.} We have 
\begin{align*} \psi(x)&=  \lim_{N\rightarrow \infty}   \int  1_{[a,x]}\cdot \Big( \sum_{k=0}^N \sum_{j=0}^{p-1}\phi\circ f^{kp+j} \Big)\ dm \\
 &=\lim_{N\rightarrow \infty}   \int  1_{[a,x]}\cdot (g\circ f^{(N+1)p} - g)\ dm \\
 &=- \int g 1_{[a,x]} \ dm+ \sum_{\lambda\in \Lambda} \int g \Phi_\lambda(1_{[a,x]}) \ dm,
\end{align*} 
The boundness of $\Phi_\lambda\colon L^1(m)\rightarrow BV$ and the assumptions on $g$ quickly  implies that  $\psi$ is absolutely continuous and  $1/2$-H\"older continuous on $I$, and it  is Lipschitz on each $S_\ell$, with $\ell \leq E$. \\

\noindent {\it $A,B \ and \ C\implies D$.} We already saw that $A$ implies $\sigma^2_m(\phi)=0$. So $D.$ follows from Theorem \ref{cohomo}.
\end{proof}

\subsection{Zygmund regularity} \label{zyg} 

Theorem \ref{vvv} tells us that the primitive $\alpha$ of the  Birkhoff sum of $\phi$ is  always Log-Lipschitz continuous.  However, Theorem \ref{lipsc} says that it is very rare that $\alpha$ has finite bounded variation. One can ask if Log-Lipschitz  regularity is sharp. Note that for expanding maps on the circle $\alpha$ is {\it always} Zygmund (see de Lima and S. \cite{central}). If each break point  is either eventually periodic  or Misiurewicz,  we can provide a definite answer for piecewise H\"older functions $\phi$.  Denote

$$\mathcal{O}^+(f,y) =\{ f^n(y)\colon \ n\in \mathbb{N}\}.$$

\begin{theorem} \label{zygmund} Let $f \in  \mathcal{B}^k_{exp}(C)$.  Suppose that
$$  \inf_{c\in \hat{C}} \inf_{x\in \mathcal{O}^+(f,c)\setminus \hat{C}} dist(x,\hat{C}) > 0$$
and let $\phi \in  \mathcal{B}^\beta(C)$, with $\beta \in (0,1)$,  be such that 
$$\int \phi \Phi_1(\gamma)=0$$
and for every $c\in C$  
\begin{itemize}
\item either $c\not\in \partial I$ and there is $N_{c^\pm} , M_{c^\pm} \in \mathbb{N}$   such that 
$$f^{M_{c^\pm}} (f^{N_{c^\pm}}(c^\pm))=f^{N_{c^\pm}}(c^\pm)$$
and
 \begin{align*} 
&\frac{1}{\ln |Df^{M_{c^+}} (f^{N_{c^+}}(c^+))| } \sum_{i=N_{c^+}}^{N_{c^+} + M_{c^+}-1}  \phi(f^i(c^+))\\
&= \frac{1}{\ln |Df^{M_{c^-}} (f^{N_{c^-}}(c^-))| } \sum_{i=N_{c^-}}^{N_{c^-} + M_{c^-}-1}  \phi(f^i(c^-)) \end{align*} 
\item or  $f^i$ is  continuous at $c$ for every $i\geq 0$ and there is $N_{c^+}=N_{c^-}$ such that $f^i(c)\not\in C$ for every $i\geq N_c$. 
\end{itemize} 
Then 
$$\psi(x)=   \int  1_{[a,x]}\cdot \Big( \sum_{k=0}^\infty  \sum_{j=0}^{p-1} \phi\circ f^{kp+j} \Big)\ dm$$
is a Zygmund function, that is, there is $C$ such that 
$$|\psi(x+h)+\psi(x-h)-2\psi(x)|\leq C|h|$$
for every $x$ such that $[x-h,x+h]\subset I$. 
\end{theorem}  
\begin{proof} Let $T$ be a multiple of  the integers  $p(f)$, $N_{c^\pm}$  and  $M_{c^\pm}$ for all $c\in C$. Let $F(x)=f^T(x)$ and $G(x)=g^T(x)$.  Then  $F\in \mathcal{B}^k_{exp}(C_F)$  and 
$$\theta(x)=\sum_{i=0}^{T-1} \phi(f^i(x)).$$
belongs to $\mathcal{B}^\beta_{exp}(C_F)$ for some finite set  $C_F$, in such way that for every $c\in C_F$

\begin{itemize}
\item[-]{\it Type I.} either we have $F^2(c^\pm)=F(c^\pm)$ and 
$$\frac{\theta(F(c^+))}{\ln |DF (F(c^+))| } = \frac{\theta(F(c^-))}{\ln |DF (F(c^-))| },$$
\item[-]{\it Type II.} or  $F^i$ is  continuous at $c$ for every $i\geq 0$ and   $F^i(c)\not\in C_F$ for every $i\geq 1$. 
\end{itemize} 
Moreover
$$\psi(x)=   \int  1_{[a,x]}\cdot \Big( \sum_{k=0}^\infty  \theta\circ F^{k} \Big)\ dm.$$
Denote
$$  d=\frac{1}{2} \inf_{c\in \hat{C}_F} \inf_{x\in \mathcal{O}^+(F,c)\setminus \hat{C}_F} dist(x,\hat{C}_F) > 0.$$
Let $x\in I$ and $h $ be such that $[x-h,x+h]\subset I$. We may assume $|h|< d$.   Then
$$\psi(x+h)+\psi(x-h)-2\psi(x)=  \int  \theta \cdot   \Big( \sum_{k=0}^\infty  L_F^{k}(1_{[x,x+h]}- 1_{[x-h,x]})  \Big)\ dm$$

Let $q$ be the smallest integer satisfying $$F^q((x-h,x+h))\cap C_F\not=\emptyset.$$
We have
$$|F^i[x-h,x+h]|\leq (\min_{z\in \hat{I}} |DF(z)|)^{-i}$$
and $supp \  L^i  (1_{[x,x+h]}- 1_{[x-h,x]}) \subset F^i[x-h,x+h]$
for $i< q$. Since 
$$\int |L^i (1_{[x,x+h]}- 1_{[x-h,x]})|\ dm \leq  |2h|, $$
$$\int L^i (1_{[x,x+h]}- 1_{[x-h,x]})|\ dm=0,$$
and $\theta$ is $\beta$-H\"older in $I\setminus C_F$ it follows that 
\begin{align} \label{bigg} 
&\sum_{i=0}^{q-1} \int \theta \cdot L^i (1_{[x,x+h]}- 1_{[x-h,x]})\ dm \nonumber \\
&=\sum_{i=0}^{q-1}  \Big( \theta(F^i(x)) \int  L^i (1_{[x,x+h]}- 1_{[x-h,x]})\ dm \nonumber \\
&+ \int (\theta-\theta(F^i(x))) L^i (1_{[x,x+h]}- 1_{[x-h,x]}) \ dm \Big) \nonumber \\
&=O(|h|\sum_{i=0}^{q-1} |F^i[x-h,x+h]|^\beta)=O(h).
\end{align}
Define $J=F^q([x-h,x+h])$.
Let $c\in C_F$ be defined by 
$$\{c\}= F^q((x-h,x+h))\cap C_F.$$
Let $J^1$ and $J^2$ be the right and left connected components of $F^q([x-h,x+h])\setminus \{c\}$ and
$$u_k = 1_{J^k} \cdot L^q( 1_{[x,x+h]}- 1_{[x-h,x]}).$$
Of course
\begin{equation} \label{pi} \int u_1 \ dm+  \int u_2\ dm=0.\end{equation} 
Fix $y\in [x-h,x+h]$ such that $F^q(y)=c$.  It follows from Lemma \ref{bvpartition} that there is $\Cll{uni}$, that depends only on $f$, such that 
$$v(u_k)\leq  \frac{\Crr{uni}}{|DF^q(y)|}.$$
and
$$ \frac{1}{|DF^q(y)|} \leq 2\Crr{dist} \frac{|h|}{|J|}.$$
Moreover
$$|u_k|\leq 2\Crr{dist} \frac{|h|}{|J|}.$$
Let $c_1=c^+$ and  $c_2=c^-$. \\

\noindent {\it Case I. $c$ is a type I point.} Let  $q_k$, with $k=1,2$ be  the smallest integer such that 
$$    |DF^{q_k}(c_k)||J^k| > \frac{d}{ \Crr{dist}}.$$
This implies that  $F^{q_k}$ is a diffeomorphism on $J^k$ and 
$$|F^{q_k}J^k|\geq \frac{d}{ \Crr{dist}^2}.$$
 Note also that $supp \ L_F^i(u_k)\subset F^i(J^k)$ and
\begin{equation}\label{con}  |F^i(J^k)|\leq ( \max_{z\in \hat{I}}  |DF(z)|)^{-i}. \end{equation} 
for every $i\leq q_k$.
 It follows from (\ref{ddis}) and (\ref{bvest}) that
 $$v(L^{q_k}(u_k))\leq  C \frac{1}{|DF^{q_k}(c_k)|}  \frac{|h|}{|J|}\leq \Cll{nova} \frac{|J^k|}{|J|} |h|.$$
 and
 $$|L^{q_k}(u_k)|\leq C \frac{1}{|DF^{q_k}(c_k)|}  \frac{|h|}{|J|}\leq \Crr{nova} \frac{|J^k|}{|J|} |h|.$$
 for some $\Crr{nova}$ that depends only on $f$. 
Consequently 
 $$|\sum_{i=0}^\infty \int \theta \cdot L^i(L^{q_1}(u_2)) \ dm|+ |\sum_{i=0}^\infty \int \theta \cdot L^i(L^{q_2}(u_2)) \ dm| \leq   \Cll{nova2}   |h|.$$

 Since  $c$ is a type I critical point we have 
$$  q_k= - \frac{\ln |J^k|}{\ln  |DF(F(c_k))|} + O(1).$$
We have that 
\begin{equation}  \label{pi1}  \int L_F^i(u_k)\ dm = \int u_k \ dm\end{equation}
and
\begin{equation} \label{pi2}  \int |L_F^i(u_k)|\ dm \leq  \int |u_k| \ dm \leq 2|h|\end{equation}
holds for every $i\leq q_k$.  Then
\begin{align} \label{fund}
&\int \theta \sum_{i=0}^{q_k}  L_F^i(u_k) \ dm  \nonumber \\ 
&=\sum_{i=0}^{q_k} \Big(    \theta(F^i(c_k))  \int L_F^i(u_k) \ dm +   \int (\theta(x) - \theta(F^i(c_k)))  L_F^i(u_k)(x)  \ dm(x)  \Big) \nonumber \\ 
&= q_k  \theta(F(c_k)) \int u_k \ dm + O(2|h|(1+\sum_{i=0}^{q_k} |F^i(J^k)|^\beta)) \nonumber\\ 
&= q_k  \theta(F(c_k)) \int u_k  \ dm + O(|h|) \nonumber\\ 
&= - \frac{ \theta(F(c_k))}{\ln  |DF(F(c_k))|}  \ln |J^k| \int u_k  \ dm + O(|h|),
\end{align} 
consequently 
\begin{align*}
&\int \theta \sum_{i=0}^{q_1}  L_F^i(u_1) \ dm + \int \theta \sum_{i=0}^{q_2}  L_F^i(u_2) \ dm  \\
&=- \frac{ \theta(F^i(c_1))}{\ln  |DF(F(c_1))|}  \Big( \ln |J^1| \int u_1  \ dm +  \ln |J^2| \int u_2  \ dm    \Big) +  O(|h|).\\
&=- \frac{ \theta(F^i(c_1))}{\ln  |DF(F(c_1))|} \Big(  \int u_1  \ dm \Big)  \Big( \ln \frac{|J^1|}{|J|}   -  \ln \Big(1 -\frac{|J^1|}{|J|}\Big)  \Big) +  O(|h|).   
\end{align*} 
If $Q_1$ and $Q_2$ are the  connected components of $J\setminus \{F^q(x)\}$  then
$$   \frac{1}{\Crr{dist}}   \leq  \frac{|Q_1|}{|Q_2|}\leq \Crr{dist},$$
so
$$\frac{1}{1+\Crr{dist}}\leq \frac{|Q_i|}{|J|}\leq  \frac{\Crr{dist}}{1+\Crr{dist}}.$$
In particular if 
\begin{equation}\label{inter} \frac{1}{1+\Crr{dist}}\leq \frac{|J^1|}{|J|}\leq  \frac{\Crr{dist}}{1+\Crr{dist}}\end{equation} 
we have
\begin{equation} \label{ooo} \Big( \int u_1  \ dm \Big)  \Big( \ln \frac{|J^1|}{|J|}   -  \ln \Big(1 -\frac{|J^1|}{|J|}\Big)\Big) =O(|h|).\end{equation} 
 if (\ref{inter}) does not hold then 
\begin{align*} &\Big| \Big( \int u_1  \ dm \Big)  \Big( \ln \frac{|J^1|}{|J|}   -  \ln \Big(1 -\frac{|J^1|}{|J|}\Big)\Big)\Big| \\
&\leq  2\Crr{dist} |h|  \min\{ \frac{|J^1|}{|J|},1 -\frac{|J^1|}{|J|}   \} \Big| \Big( \ln \frac{|J^1|}{|J|}   -  \ln \Big(1 -\frac{|J^1|}{|J|}\Big)\Big)\Big| \\
&\leq 2\Crr{dist} |h|  \sup_{0< t< 1}   \min\{t,1 -t   \}|\ln t - \ln (1-t)|,  \end{align*} 
and  (\ref{ooo}) holds.  So in every  case
\begin{equation}\label{e1}\int \theta \cdot \sum_{i=0}^{q_1}  L_F^i(u_1) \ dm + \int \theta \cdot \sum_{i=0}^{q_2}  L_F^i(u_2) \ dm=O(h).\end{equation}

\noindent {\it Case II. $c$ is a type II point.} Suppose $|J^1|\geq  |J^2|$ (the other case is analogous).   Let  $q_1$  be  the smallest integer such that 
$$ ( \max_{z\in \hat{I}}  |DF(z)|)  |DF^{q_1-1}(F(c))||J^1| > \frac{d}{ \Crr{dist}}.$$ 
Then $F^{q_1}$ is a diffeomorphism on $J^1$ and $J^2$. Let $q_2=q_1$. 
It follows from (\ref{ddis}) and (\ref{bvest}) that
 $$v(L^{q_k}(u_k))\leq  C \frac{1}{|DF^{q_k}(c_k)|}  \frac{|h|}{|J|}\leq \Cll{novad} \frac{|J^1|}{|J|} |h|.$$
 and
 $$|L^{q_k}(u_k)|\leq C \frac{1}{|DF^{q_k}(c_k)|}  \frac{|h|}{|J|}\leq \Crr{novad} \frac{|J^1|}{|J|} |h|.$$
for some $\Crr{novad}$ that depends only on $f$ and $i=1,2$. We conclude that 
\begin{equation} \label{tail} |\sum_{i=0}^\infty \int \theta \cdot L^i(L^{q_1}(u_1)) \ dm|+ |\sum_{i=0}^\infty \int \theta \cdot L^i(L^{q_2}(u_2)) \ dm| \leq   \Cll{nova22}   |h|.\end{equation}

Note also that  $supp \ L_F^i(u_k)\subset F^i(J^k)$ and (\ref{con}), (\ref{pi1}), (\ref{pi2}) hold for $i\leq q_k$. Consequently
\begin{align*}
&\int \theta \sum_{i=0}^{q_k}  L_F^i(u_k) \ dm \\ 
&=  \sum_{i=0}^{q_k} \Big(    \theta(F^i(c_k))  \int L_F^i(u_k) \ dm   +    \int (\theta(x) - \theta(F^i(c_k)))  L_F^i(u_k)(x)  \ dm(x)  \Big) \\ 
&=  \theta(c_k) \int u_k \ dm +   \sum_{i=1}^{q_k}   \theta(F^i(c))  \int u_k \ dm +   O(|h|) \\ 
&=  \sum_{i=1}^{q_k}   \theta(F^i(c))  \int u_k \ dm +   O(|h|).\\ 
\end{align*} 
and  (\ref{pi}) implies
\begin{align*}
&\int \theta \sum_{i=0}^{q_1}  L_F^i(u_1) \ dm + \int \theta \sum_{i=0}^{q_2}  L_F^i(u_2) \ dm \\ 
&=  \sum_{i=1}^{q_1}   \theta(F^i(c)) (  \int u_1 \ dm + \int u_2 \ dm)    +   O(|h|).\\  
&=O(|h|).
\end{align*} 
To conclude the proof of the theorem, note that 
\begin{align} \label{broke} &\psi(x+h)+\psi(x-h)-2\psi(x)\\
&=  \int  \theta \cdot   \Big( \sum_{i=0}^{q-1}  L_F^{i}(1_{[x,x+h]}- 1_{[x-h,x]})  \Big)\ dm \nonumber \\
&+  \int  \theta \cdot   \sum_{i=0}^{q_1-1}  L_F^{i}(u_1) \ dm +  \int  \theta \cdot   \sum_{i=0}^{q_2-1}  L_F^{i}(u_2) \ dm   \nonumber \\
&+ \int  \theta \cdot  \sum_{i=0}^{\infty}  L_F^{i}(L_F^{q_1}u_1) \ dm +  \int  \theta \cdot  \sum_{i=0}^{\infty}  L_F^{i}(L^{q_2}u_2) \ dm   \nonumber   \end{align} 
and apply the previous estimates.
\end{proof} 

\begin{theorem} Let $f \in  \mathcal{B}^k_{exp}(C)$ be a piecewise expanding map.  Suppose that there exists $c\in C\setminus \partial I$ satisfying 

$$  \inf_{x\in \mathcal{O}^+(f,c^\pm)\setminus \hat{C}} dist(x,\hat{C}) > 0$$
and there  are $N_{c^\pm} , M_{c^\pm} \in \mathbb{N}$  with $f^{M_{c^\pm}} (f^{N_{c^\pm}}(c^\pm))=f^{N_{c^\pm}}(c^\pm)$. Let $\phi \in  \mathcal{B}^\beta(C)$, with $\beta \in (0,1)$, such that 
$$\int \phi \Phi_1(\gamma)=0$$
for every $\gamma\in BV$ but 
\begin{align*} 
&\frac{1}{\ln |Df^{M_{c^+}} (f^{N_{c^+}}(c^+))| } \sum_{i=N_{c^+}}^{N_{c^+} + M_{c^+}-1}  \phi(f^i(c^+))\\
&\not= \frac{1}{\ln |Df^{M_{c^-}} (f^{N_{c^-}}(c^-))| } \sum_{i=N_{c^-}}^{N_{c^-} + M_{c^-}-1}  \phi(f^i(c^-)) \end{align*} 
Then 
$$\psi(x)=   \int  1_{[-1,x]}\cdot \Big( \sum_{k=0}^\infty  \sum_{j=0}^{p-1} \phi\circ f^{kp+j} \Big)\ dm$$
is not a Zygmund function.
\end{theorem} 
\begin{proof} Let $T$ be a common multiple of  the integers  $p(f)$, $N_{c^\pm}$  and  $M_{c^\pm}$. Let $F(x)=f^T(x)$. Define $\theta$ as in the proof of Theorem \ref{zygmund}. Then we have $F^2(c^\pm)=F(c^\pm)$ and 
$$\frac{\theta(F(c^+))}{\ln |DF (F(c^+))| } \neq  \frac{\theta(F(c^-))}{\ln |DF (F(c^-))| }.$$
Let 
$$  d=\frac{1}{2} \inf_{x\in \mathcal{O}^+(F,c)\setminus \hat{C}_F} dist(x,\hat{C}_F) > 0.$$
Using the same notation as in the proof of Theorem \ref{zygmund} take $x=c$ and $0< h< d$. Then $q=0$, $J=[c-h,c+h]$, $c_1=c^+$, $c_2=c^-$, $J^1=[c,c+h]$, $J^2=[c-h,c]$,  $u_1=1_{[c,c+h]}$ and $u_2=-1_{[c-h,c]}$  and we can write (\ref{broke}). Since  (\ref{pi}),  (\ref{tail}) and (\ref{fund}) holds 
\begin{align*} &\psi(c+h)+\psi(c-h)-2\psi(c)\\
&=\Big(  \frac{ \theta(F(c^-))}{\ln  |DF(F(c^-))|}   - \frac{ \theta(F(c^+))}{\ln  |DF(F(c^+))|} \Big)  |h|  \ln |h|   \ dm 
                  + O(|h|), \end{align*} 
so $\psi$ is not  Zygmund. 
\subsection{Invariant distributions which are not measures}
An interesting application of Birkhoff sums as distributions is the construction of distributions  in $\Theta\in BV^\star$ which are {\it invariant } with respect to a certain piecewise expanding map $f\in \mathcal{B}^k_{exp}(C)$, that is
$$\Theta(g)=\Theta(g\circ f)$$
for every $g\in BV$, but that are {\it not}  signed measures. 
Choose  $\phi \in \mathcal{B}^\beta(C)\cap BV$ such that 
$$\int \phi   \Phi_1(\gamma) \ dm=0$$
for every $\gamma\in BV$. Consider the distribution $\Theta_\phi\in BV^\star$ given by
$$\Theta_\phi(g)=\sigma_m(g,\overline{\phi})=\sum_\ell m(A_\ell) \sigma_{\mu_\ell} (g,\overline{\phi}),$$
where $\sigma_m$ is the hermitian form defined in Theorem \ref{cohomo}.

The following result  tells us that it is quite rare that $\Theta_\phi$ is a signed measure. 
\begin{theorem}\label{xcxc} The  functional $\Theta_\phi\in BV^\star$ is a $f$-invariant distribution. The  following statements are equivalent
\begin{itemize}
\item[A.]  $\Theta_\phi$ is a signed measure, that is,  there is a signed regular measure $\nu$ such that  for every $\psi\in BV\cap C^0(I)$
$$\Theta_\phi(\psi)=\int \psi \ d\nu.$$
\item[B.] $\phi=\psi\circ f -\psi$, where $\psi\in L^2(m)$ and $\psi\in L^\infty(S_\ell)$ for every $\ell\leq E$.
\item[C.]  $\Theta_\phi=0$.
\end{itemize}
Moreover $A.-C.$ implies
\begin{itemize}
\item[D.]  We have that 
\begin{equation}\label{inficoddd} \sum_{j=0}^{M-1} \phi(f^j(q))=0\end{equation} 
holds for every $M$ and $q\in \hat{S}_\ell$, with $\ell\leq E$,  such that $f^M(q)=q$.
\end{itemize} 
Furthermore if $f$ is markovian, $p(f)=1$ and it has an absolutely continuous ergodic invariant probability whose support is $I$ then $D.$ is equivalent to $A.-C.$
\end{theorem}

\begin{proof} Note that $\sigma_m^2(g\circ f - g)=0$. By the Cauchy-Schwarz inequality
$$|\sigma_m(g\circ f,\overline{\phi})- \sigma_m(g,\overline{\phi})|^2=|\sigma_m(g\circ f - g,\phi)|^2\leq \sigma_m^2(g\circ f - g)\sigma_m^2(\overline{\phi})=0,$$
so $\Theta_\phi$ is $f$-invariant.  Using exactly the same argument as in Broise \cite[Chapter 6]{broise} we obtain
\begin{align*}  &\sigma_{\mu_{\ell_0}}(\gamma,\overline{ \phi})\\
&=  \lim_{N \rightarrow \infty} \frac{1}{N}  \int \Big(\sum_{i=0}^{N-1} \gamma \circ f^i  \Big)\Big(\sum_{i=0}^{N-1} \overline{ \phi} \circ f^i  \Big)   \ \rho_{\ell_0} \ dm\\
&=-\int \gamma\overline{ \phi} \rho_{\ell_0} \ dm +\lim_{N} \sum_{i=0}^N (1-\frac{j}{N})\Big(  \int \gamma\overline{ \phi}\circ f^j \rho_{\ell_0} \ dm+ \int \gamma\circ f^j\overline{ \phi} \rho_{\ell_0} \ dm \Big)\\
&=-\int \gamma\overline{\phi} \rho_{\ell_0} \ dm  + \sum_{\lambda\in \lambda_1\setminus \{1\} }\frac{1}{1-\lambda} \int \Phi_\lambda(\gamma \rho_{\ell_0})\overline{ \phi}  \ dm + \sum_{i=0}^\infty \int K^i(\gamma \rho_{\ell_0})\overline{ \phi} \ dm \\
&+ \sum_{\lambda\in \lambda_1\setminus \{1\} }\frac{1}{1-\lambda} \int \Phi_\lambda(\overline{ \phi}) \gamma \rho_{\ell_0}  \ dm + \sum_{i=0}^\infty \int K^i(\overline{ \phi}) \gamma \rho_{\ell_0} \ dm
\end{align*}
for every $\gamma\in BV$, $\ell_0\leq E$. 
One can see that 
$$\gamma \mapsto \sigma_{\mu_{\ell_0}}(\gamma,\overline{ \phi}) -  \sum_{i=0}^\infty \int K^i(\gamma \rho_{\ell_0}) \overline{ \phi}\ dm$$
is a bounded functional considering the $L^1(m)$ norm in its domain $BV$. Since $K$ is a contraction on $BV$ and $BV$ is a Banach algebra it follows that 
 $$\sum_{i=0}^\infty \int K^i(\gamma \rho_{\ell_0}) \overline{ \phi}\ dm$$
 belongs to $BV^\star$, so consequently
 $$\gamma \mapsto \sigma_{m}(\gamma,\overline{ \phi}) $$
 is in  $BV^\star$. 
 By  Theorem \ref{vvv}.B we have that there are $\Cll{o11},\Cll{o22}$ such that
$$\sigma_{\mu_{\ell_0}}(\gamma,\overline{ \phi}) \leq  \Crr{o11}((\ln |\gamma|_{BV}-  \ln |\gamma|_{L^1(m)})+\Crr{o22})|\gamma|_{L^1(m)}  $$
for every $\ell_0\leq E$, so 
\begin{equation}\label{o111}  \sigma_{m}(\gamma,\overline{ \phi}) \leq  \Crr{o11}((\ln |\gamma|_{BV}-  \ln |\gamma|_{L^1(m)})+\Crr{o22})|\gamma|_{L^1(m)}.\end{equation} 

\noindent Of course $C.\implies A.$ \\

\noindent {$A.\implies B. \ and \ D. $} Suppose that  $\Theta_\phi$ is a signed measure, that is, there is a signed regular measure $\nu$ such that  for every $\psi\in BV\cap C^0(I)$
$$\Theta_\phi(\psi)=\int \psi \ d\nu.$$
Choosing  $\gamma=1_{[x,y]}$ it is easy  to see that there is a sequence  $\gamma_k\in BV\cap C^0(I)$ such that $\sup_k |\gamma_k|_{BV} < \infty$, $\lim_k \gamma_k(z)= 1_{[x,y]}(z)$ for every $z$. In particular  $\lim_k |\gamma -\gamma_k|_{L^1(m)}=0$ and  consequently (\ref{o111}) implies
 $$\Theta_\phi(1_{[x,y]})=\lim_k \theta_\phi(g_k)=\lim_k \int g_k \ d\nu= \int 1_{[x,y]} \ d\nu.$$
A similar argument shows that 
$$y\mapsto \Theta_\phi(1_{[x,y]})$$
is continuous. So $\nu$ does not have atoms and $ \Theta_\phi(1_{[x,y]})$ is continuous and it has  bounded variation with respect to $y$. 

Given $\gamma\in BV$  there is a sequence  $\gamma_k\in BV\cap C^0(I)$ such that $\sup_k  |\gamma_k|_{BV} < \infty$ and  $\lim_k \gamma_k(z)= \gamma(z)$ for every $z$ except for $z$ in the set of discontinuities of $\gamma$, that is countable. Using an argument similar to the argument with $\gamma=1_{[x,y]}$, and using that $\nu$ does not have atoms,  we can conclude that
$$\Theta_\phi(\gamma)=\int \gamma \  d\nu$$
for every $\gamma\in BV$.

In particular for  each $x\in S_{\ell_0}$, with $\ell_0\leq E$  and $y$ such that $[x,y]\subset S_{\ell_0}$ we can choose 
$\gamma =1_{[x,y]}\rho_{\ell_0}^{-1}\in BV$ and we can prove that 
$$y\mapsto \Theta_\phi(1_{[x,y]}\rho_{\ell_0}^{-1})$$
is continuous and it has bounded variation.  We have
$$\Theta_\phi(1_{[x,y]}\rho_{\ell_0}^{-1})=m(A_{\ell_0}) \sigma_{\mu_{\ell_0}}(1_{[x,y]}\rho_{\ell_0}^{-1}, \phi),$$
so  $y\mapsto \sigma_{\mu_{\ell_0}}(1_{[x,y]}\rho_{\ell_0}^{-1}, \phi) $ is a bounded variation function on an interval. One can see that 
$$y\mapsto \sigma_{\mu_{\ell_0}}(1_{[x,y]}\rho_{\ell_0}^{-1},\overline{ \phi}) -  \sum_{i=0}^\infty \int K^i(1_{[x,y]}) \overline{\phi}   \ dm$$
is a Lipschitz function, so we conclude that 
$$y\mapsto u(y)=\sum_{i=0}^\infty \int K^i(1_{[x,y]}) \overline{\phi} \ dm= \sum_{N=0}^\infty \sum_{j=0}^{p-1} \int L^{Np+j}(1_{[x,y]}) \overline{\phi}  \ dm  $$
has  bounded variation. By Theorem \ref{lipsc}  this  implies that (\ref{inficoddd}) holds for every $q$ and $m$ such that  $f^m(q)=q$.\\

\noindent {\it $B. \implies C.$}  Note that if $B.$ holds, then $\sigma^2_m(\phi)=\sigma^2_m(\overline{\phi})=0$ and the Schwartz inequality for the hermitian form $\sigma$ implies
$$|\Theta_\phi(g)|=|\sigma_m(g,\overline{\phi})|\leq \sigma_m(g)\sigma_m(\overline{\phi})=0,$$
so $C.$ holds.\\

\noindent {\it $D. \implies A., B., C.$ under additional assumptions} The markovian property  of $f$ and the mixing property of the unique invariant absolutely continuous probability  implies that  if $D.$ holds we can use the Livisc-type result for subshifts of finite type as in Parry and Pollicott \cite[Proposition 3.7]{pp}  to conclude that there is a piecewise H\"older continuous function $\psi$ satisfying  $\phi=\psi\circ f -\psi.$ on $I$. So $B$ holds. 
\end{proof}

\bibliographystyle{abbrv}
\bibliography{bibliografia}

\end{document}